\documentclass[11pt,a4paper]{article}
\usepackage{mathrsfs, amsmath, amsthm, amssymb, bm, mathtools, thm-restate}
\usepackage{authblk}

\usepackage{graphicx, xcolor, tikz}
\usetikzlibrary{positioning}
\usetikzlibrary{calc, shapes}
\usepackage{caption, subcaption}
\usepackage{array}
\usepackage{cases}
\usetikzlibrary{calc}
\usetikzlibrary{decorations.markings}
\usetikzlibrary{arrows}
\usetikzlibrary{patterns}
\usetikzlibrary{shapes,snakes}
\pgfdeclarelayer{foreground}
\pgfsetlayers{main,foreground}
\usepackage{pgfmath}
\usepackage{caption}
\usepackage{thm-restate}
\makeatletter
\makeatletter
\def\th@plain{%
	\upshape 
}
\makeatother

\makeatletter
\renewenvironment{proof}[1][\proofname]{\par
	\pushQED{\qed}%
	\normalfont \topsep6\p@\@plus6\p@\relax
	\trivlist
	\item[\hskip\labelsep
	\bfseries
	#1\@addpunct{.}]\ignorespaces
}{%
	\popQED\endtrivlist\@endpefalse
}
\makeatother

\newtheorem{theorem}{Theorem}
\newtheorem{lemma}[theorem]{Lemma}
\newtheorem{corollary}[theorem]{Corollary}
\theoremstyle{definition}
\newtheorem{definition}[theorem]{Definition}

\newtheorem{observation}[theorem]{Observation}

\newtheorem{claim}{Claim}

\usepackage[left=25mm, top=25mm, bottom=25mm, right=25mm]{geometry}
\usepackage[T1]{fontenc}
\usepackage[inline]{enumitem}
\usepackage[square, numbers, sort&compress]{natbib}
\usepackage[pdftex,%
bookmarks=true,%
bookmarksnumbered=true, 
bookmarksopen=true, 
plainpages=false,%
pdfpagelabels,%
colorlinks=true, 
linkcolor=blue, 
citecolor=blue,%
anchorcolor=green,
urlcolor=black,
breaklinks=true,
hyperindex=true]{hyperref}
\usepackage{cleveref}
\Crefname{claim}{Claim}{Claims}
\Crefname{observation}{Observation}{Observations}

\definecolor{DodgerBlue}{RGB}{30,144,255}
\definecolor{SkyBlue}{RGB}{135,206,235}
\definecolor{DeepSkyBlue}{RGB}{0,191,255}
\definecolor{MidnightBlue}{RGB}{25,25,112}

\title{Brooks' theorem for signed graphs with $\Delta=3$}

\author[1,2]{Reza Naserasr}
\author[1,2]{Huan Zhou \thanks{Corresponding author.}}
\affil[1]{\small Universit\'e Paris Cit\'e, CNRS, IRIF, F-75013, Paris, France. {Emails: \texttt{\{reza, zhou\}@irif.fr}}}
\affil[2]{\small Zhejiang Normal University, Jinhua, China. {Email: \texttt{huanzhou@zjnu.edu.cn}}}

\begin{document}
	\maketitle

	\begin{abstract}
Circular $r$-coloring of a signed graph $(G,\sigma)$ is a mapping of its vertices to a circle of circumference $r$ such that: I. each pair of vertices with a negative connection is at distance at least $1$, and II. for each pair with a positive connection, the distance of one from the antipodal of the other is at least $1$. A signed graph $(G,\sigma)$ admits a circular $r$-coloring for some values of $r$ if and only if it has no negative loop. The smallest value of such $r$ is the circular chromatic number, denoted  $\chi_{c}(G,\sigma)$.
The circular chromatic number is a refinement of the balanced chromatic number, which is mostly studied under the equivalent term $0$-free coloring in the literature. 

Extending Brooks' theorem,  M\'a\v cajov\'a,  Raspaud, and \v{S}koviera showed that if $\Delta(G)$ is an even number, $G$ is connected, and $(G,\sigma)$ is not (switching) isomorphic to $(K_{\Delta+1},-)$ or $C_{-\ell}$ (when $\Delta(G)=2$), then $\chi_c(G,\sigma)\leq \Delta(G)$ and that the upper bound is tight. For the odd values of $\Delta(G)$, assuming a connected signed graph $(G,\sigma)$ is not  isomorphic to $(K_{\Delta+1},-)$, determining the best upper bound for $\chi_c(G, \sigma)$  proves to be more of a challenge. In this work, addressing the first step of this question, we show that if $(G, \sigma)$ is a signed graph of maximum degree 3 with no  component  isomorphic to $(K_4, -)$, then $\chi_{c}(G, \sigma)\leq \frac{10}{3}$. The upper bound is tight even among signed cubic graphs of girth 5. In particular, there is a signature on the Petersen graph for which the upper of $\frac{10}{3}$ is achieved. 
	\end{abstract}
	\section{Introduction}
	
	A \emph{signed graph}  $(G,\sigma)$ is a graph $G=(V,E)$ together with an assignment $\sigma:E(G)\rightarrow \{+,-\}$. When $\sigma$ is of little importance, we may write $\widehat{G}$ in place of $(G,\sigma)$. An edge with sign $-$ ($+$, respectively) is a \emph{negative edge} (\emph{positive edge}). We may write $(G,-)$ if $\sigma$ assigns $-$ to all edges. A digon, denoted $\widetilde{K}_2$, is the signed multigraph with two vertices and two parallel edges connecting them of which one is positive and the other is negative. In a signed graph, the sign of a (closed) walk is defined as the product of the signs of the edges in that walk. 
	
	In a signed graph $(G, \sigma)$, \emph{switching} a vertex $v$ means to reverse the sign of each edge incident to $v$, noting that the sign of a loop must be reversed twice. More generally, for a vertex subset $X \subseteq V(G)$, switching $X$ is to switch each vertex in $X$. That is equivalent to reversing the signs of all edges in the edge-cut $E(X, V(G) \backslash X)$. Observe that the sign of a closed walk does not change after a switching. 
	Two signed graphs are said to be \emph{switching equivalent} if one can be obtained from the other by a switching a subset $X$ of vertices. A fundamental Theorem of signed graphs, proved by Zaslavsky \cite{zaslavsky82}, is that two signed graphs $(G, \sigma)$ and $(G, \sigma')$ are switching equivalent if and only if they have the same set of positive cycles. 	
	
	An edge-sign preserving homomorphism of a signed graph $(G, \sigma)$ to a signed graph $(H, \pi)$ is a mapping $f: V(G) \rightarrow V(H)$ such that for every edge $uv$ of $(G, \sigma)$, $f(u)f(v)$ is an edge of $(H, \pi)$, and $\sigma(uv)=\pi(f(u)f(v))$. We  write $(G, \sigma) \xrightarrow{s. p.}(H, \pi)$ if there exists an edge-sign preserving homomorphism from $(G, \sigma)$ to $(H, \pi)$. 
	A signed graph $(G, \sigma)$ is said to admit a homomorphism to $(H,\pi)$ if a signed graph $\left(G, \sigma^{\prime}\right)$ which is switching equivalent to $(G, \sigma)$ admits a singed preserving homomorphism to $(H, \pi)$. When there exists a homomorphism we write $(G, \sigma) \rightarrow (H, \pi)$. For more on homomorphisms of signed graphs we refer to \cite{NRS15} and \cite{NSZ21}. An \emph{isomorphism} between $(G, \sigma)$ and $(H, \pi)$ is a homomorphism of one to the other which is a bijection. When there is an isomorphism between the two signed graphs, we write $(G, \sigma) \cong (H, \pi)$.

	For integers $p \geqslant 2 q>0$ such that $p$ is even, the \emph{signed circular clique} $K_{p ; q}^s$ has vertex set $\{0,1, \ldots, p-1\}$, in which $ij$ is a negative edge if and only if $q \leqslant|i-j| \leqslant p-q$, and $ij$ is a positive edge if and only if either $|i-j| \leqslant \frac{p}{2}-q$ or $|i-j| \geqslant \frac{p}{2}+q$. Note that in $K_{p ; q}^s$, each vertex $i$ is incident to a positive loop.  
	
	Among the various equivalent definitions of the circular chromatic number of a signed graph given in \cite{NWZ21}, we adopt the following in this work, but noting that for suitable connection with the theory of balanced coloring (see \cite{JMNNQ25} and references therein for more details), we have swapped the conditions imposed by the negative and the positive connections.

	\begin{definition}
		A signed graph $(G, \sigma)$ with no negative loop admits a circular $\frac{p}{q}$-coloring, where $p$ and $q$ are positive integers with $p$ being even and satisfying $p\geq 2q$, if  $(G, \sigma) \xrightarrow{s . p .} K_{p ; q}^s$. The circular chromatic number of $(G, \sigma)$, denoted $\chi_c(G, \sigma)$, is defined as  
		
		$$
		\chi_c(G, \sigma)=\inf \{\frac{p}{q} \mid (G, \sigma) \xrightarrow{s . p .} K_{p ; q}^s\}.
		$$
	\end{definition}
	
	Observe that in $K_{p ; q}^s$, the vertices $i$ and $i+\frac{p}{2}$ are \emph{anti-twin}, meaning if a vertex $j$ is adjacent to $i$ with the sign $\sigma(ij)$, then it is adjacent to $i+\frac{p}{2}$ with the sign $-\sigma(ij)$. Let $\widehat{K}_{p ; q}^s$ be the signed subgraph of $K_{p ; q}^s$ induced by vertices $\left\{0,1, \ldots, \frac{p}{2}-1\right\}$. Observe that switching $\left\{\frac{p}{2}, \frac{p}{2}+1, \ldots, p-1\right\}$ and then identifying each vertex $i$ with $i+\frac{p}{2}$ is a homomorphism of $K_{p ; q}^s$ to $\widehat{K}_{p ; q}^s$. Thus the circular chromatic number of $(G, \sigma)$ can be equivalently defined as: 
	$$
	\chi_c(G, \sigma)=\inf \left\{\frac{p}{q} \mid p (G, \sigma) \rightarrow \widehat{K}_{p ; q}^s\right\}. 
	$$
	
	\subsection{Brooks type theorems}
	Considering the structure of $K_{2p ; 1}^s$, it follows from the definitions that a signed graph $(G,\sigma)$ admits a circular $2p$-coloring if and only if there is an assignment $c$ of colors $\{\pm 1, \pm 2, \ldots, \pm p\}$ to the vertices such that for each edge $e=xy$, we have $c(x)\neq - \sigma(e)c(y)$. In this notation, $-i$ represents $i+p$. This coloring has the property that the vertices which are assigned colors $i$ or $-i$ induce a balanced set, i.e., they induce no negative cycle. For more on this subject we refer to \cite{JMNNQ25}.  
	
	Thus if we consider circular $2p$-colorings of $(G, -\sigma)$, then the condition for $c$ to be a circular coloring is that $c(x)\neq \sigma(e)c(y)$.  That is equivalent to being a 0-free $p$-coloring of $(G, \sigma)$, introduced in \cite{zaslavsky82} and further studied as a proper coloring of signed graphs. In particular a result of \cite{MRS16} in extending Brooks' theorem to 0-free colorings of signed graphs can be restated as follows. 
	
	\begin{theorem}\label{thm:BrooksDelat=2k}
		If $(G, \sigma)$ is a connected signed simple graph of maximum degree at most $p$ where $p$ is a positive even integer, $p\geq 4$, unless $(G, \sigma)$ is (switching) isomorphic to $(K_{p+1}, -)$, it admits a circular $p$-coloring. Moreover, the upper bound of $p$ is tight. 
	\end{theorem}
	
	The introduction of circular coloring then raises the question of finding the best possible upper bound for the circular chromatic number of signed simple graphs having maximum degree $2k+1$. The question proves to be of higher difficulty compared to the case of even maximum degree. In this work, addressing the first case of the question, we prove the following.
	
	\begin{theorem}\label{thm:BrooksDelat=3}
		If $(G, \sigma)$ is a connected signed simple graph of maximum degree 3 which is not (switching) isomorphic to $(K_4,-)$, then $\chi_c(G, \sigma) \leq \frac{10}{3}$.
	\end{theorem}
	
	We note that, in analogy to the Brooks' theorem, negative cycles are the minimal signed graphs not admitting circular 2-coloring. Thus for signed graphs of maximum degree 2 ($p=2$), the negative girth, that is the length of smallest negative cycle, determines the circular chromatic number: $\chi_{c}(C_{-\ell})=2+\frac{2}{\ell-1}$, where $C_{-\ell}$ is a negative cycle of length $\ell$. If there is no negative cycle, then circular chromatic number is 2.
	
	\subsection{Critical signed graphs}
	
	A typical approach for proving \Cref{thm:BrooksDelat=3} is to work with a minimal counterexample. Observe that if a signed graph $(G, \sigma)$ does not admit a circular $r$-coloring, then it has a minimal subgraph which does not admit circular $r$-coloring. This leads to the notion of critical signed graphs with respect to circular coloring.
	
	\begin{definition}
		A signed graph $(G, \sigma)$ is said to be \emph{circular $r$-critical}, or simply $r$-critical, if it does not admits a circular $r$-coloring, but each of  its proper subgraphs admits a circular $r$-coloring.
	\end{definition}

	More generally, given signed graphs $(G, \sigma)$ and $(H, \pi)$, we say $(G, \sigma)$ is $(H, \pi)$-critical if $(G, \sigma)$ admits no homomorphism to $(H, \pi)$, but every proper subgraph of it does.  Then a $\frac{p}{q}$-critical signed graph is the same as a $\widehat{K}_{p;q}^s$-critical signed graph. 
	
	Since  $\widehat{K}_{p ; q}^s$ is vertex transitive, every $r$-critical signed graph is 2-connected. Furthermore, it contains no parallel edges of the same sign. \Cref{thm:BrooksDelat=3} is an immediate corollary of the following, which is a key contribution of this work.
	
	\begin{theorem}\label{thm:10/3Critical}
		If $(G,\sigma)$ is a 10/3-critical signed graph that is isomorphic neither to $\widetilde{K}_2$ nor to $(K_4,-)$, then $$|E(G)|\ge \frac{3|V(G)|+1}{2}.$$
	\end{theorem}

	\begin{corollary}
		If $\widehat{G}$ is a subcubic signed simple graph, then $\chi_c(\widehat{G})\leq 10/3$.
	\end{corollary}
	
	We use the potential method of Kostochka and Yancey \cite{KY14, KY14JCTB} to prove \Cref{thm:10/3Critical}. Assume the potential function of signed graph $(G,\sigma)$ is $p(\widehat{G})=3|V(\widehat{G})|-2|E(\widehat{G})|$. 
	
	\section{The 10/3-circular coloring and homomorphism}
	
	Before proceeding to the main proof, we first examine certain properties of 10/3-coloring and 10/3-critical graphs, including colorings of specific substructures and partial colorings. These preliminary results will facilitate coloring extension in the subsequent inductive arguments.

	Assume $\widehat{G}$ is a $10/3$-colorable signed graph. In other words, there is an edge-sign preserving homomorphism from $\widehat{G}$ to $K_{10;3}^{s}$ (\Cref{fig:circular_clique}(a)), and there is a (switching) homomorphism from $\widehat{G}$ to $\widehat{K}_{10;3}^{s}$ (\Cref{fig:circular_clique}(b)). In our figures red (and solid) edges are negative edges and blue (and dashed) edges are positive.
	
	\begin{figure}[h!]
		\centering
		\begin{subfigure}[t]{0.45\textwidth}
			\centering
			\begin{tikzpicture}[rotate=90, scale=0.6]
				\def\n{10}
				\def\radius{3}
				\foreach \i in {1,...,\n} {
					\coordinate (P\i) at ({-360/\n * (\i - 1)}:\radius);
				}
				
				\foreach \i in {1,...,\n} {
					\pgfmathtruncatemacro\next{mod(\i, \n) + 1}
					\draw[blue, dash pattern=on .41mm off .41mm] (P\i) -- (P\next);
				}
				
				\foreach \i in {1,...,\n} {
					\pgfmathtruncatemacro\target{mod(\i+1, \n) + 1}
					\draw[blue, dash pattern=on .41mm off .41mm] (P\i) -- (P\target);
				}
				
				\foreach \i in {1,...,\n} {
					\pgfmathtruncatemacro\target{mod(\i+2, \n) + 1}
					\draw[red] (P\i) -- (P\target);
				}
				
				\foreach \i in {1,...,\n} {
					\pgfmathtruncatemacro\target{mod(\i+3, \n) + 1}
					\draw[red] (P\i) -- (P\target);
				}
				
				\foreach \i in {1,...,\n} {
					\pgfmathtruncatemacro\target{mod(\i+4, \n) + 1}
					\draw[red] (P\i) -- (P\target);
				}
				
				\foreach \i in {1,...,5} {
					\node[black] at ({-360/\n * (\i - 1)}:\radius + 0.5) {\i};
				}
				
				\foreach \i in {6,...,10} {
					\pgfmathtruncatemacro\label{-(\i-5)}
					\node[black] at ({-360/\n * (\i - 1)}:\radius + 0.5) {\label};
				}
				
				\foreach \i in {1,...,\n} {
					\pgfmathsetmacro\angle{-360/\n * (\i - 1)}
					\path let \p1 = (P\i) in
					coordinate (C1) at ($ (P\i) + (\angle+50:2.5) $)
					coordinate (C2) at ($ (P\i) + (\angle-50:2.5) $);
					\draw[blue, dash pattern=on .41mm off .41mm]
					(P\i) .. controls (C1) and (C2) .. (P\i);
				}
			\end{tikzpicture}
			\caption{$K_{10;3}^s$}
		\end{subfigure}
		\hfill
		\begin{subfigure}[t]{0.45\textwidth}
			\centering
			\begin{tikzpicture}[rotate=90,scale=0.6]
				\def\n{5}
				\def\radius{3}
				\foreach \i in {1,...,\n} {
					\coordinate (P\i) at ({-360/\n * (\i - 1)}:\radius);
				}
				\foreach \i in {1,...,\n} {
					\pgfmathtruncatemacro\next{mod(\i, \n) + 1}
					\draw[blue, dash pattern=on .41mm off .41mm] (P\i) -- (P\next);
				}
				\foreach \i in {1,...,\n} {
					\pgfmathtruncatemacro\target{mod(\i+2, \n) + 1}
					\draw[red] (P\i) -- (P\target);
				}
				
				
				\foreach \i in {1,...,5} {
					\node[black] at ({-360/\n * (\i + 4)}:\radius + 0.5) {
						\ifnum\i=1 1
						\else\ifnum\i=2 3
						\else\ifnum\i=3 5
						\else\ifnum\i=4 -2
						\else\ifnum\i=5 -4
						\fi\fi\fi\fi\fi
					};
				}

				\foreach \i in {1,...,\n} {
					\pgfmathsetmacro\angle{-360/\n * (\i - 1)}
					\path let \p1 = (P\i) in
					coordinate (C1) at ($ (P\i) + (\angle-50:2.5) $)
					coordinate (C2) at ($ (P\i) + (\angle+50:2.5) $);
					\draw[blue, dash pattern=on .41mm off .41mm]
					(P\i) .. controls (C1) and (C2) .. (P\i);
				}
			\end{tikzpicture}
			\caption{$\widehat{K}_{10;3}^s$}
		\end{subfigure}
		
		\caption{Circular clique structures}
		\label{fig:circular_clique}
	\end{figure}
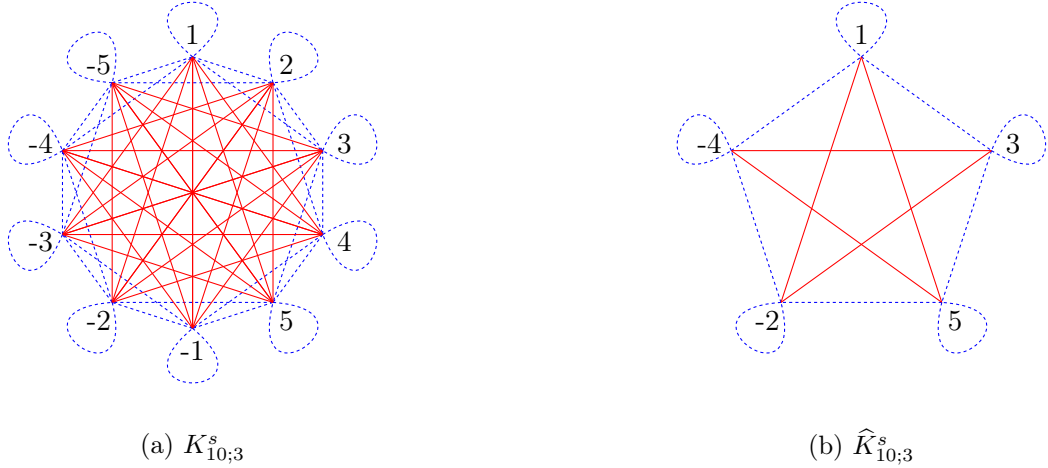

	In approach with potential technique we need the smallest subgraph of $\widehat{K}_{10;3}^s$ that forces the full structure of it. The subgraph  $\widehat{T}$ of \Cref{fig:TandT+} is identified as such in the sense of the following lemma. 	
	\begin{lemma}\label{lem:T}
		Every homomorphism of $\widehat{T}$ to $\widehat{K}^s_{10;3}$ is surjective.  
	\end{lemma}
	
	\begin{proof}
		Let $\phi$ be a mapping of $\widehat{T}$ to $\widehat{K}^s_{10;3}$. There are only two pairs of vertices each of whose identification does not create a loop or a digon: $\{v_1,v_2\}$, and $\{v_1,v_5\}$. All other pairs are either in a negative triangle or are the diagonal vertices of a negative 4-cycle. However, identification one of these two pairs would result in a subgraph switching equivalent to $(K_4,-)$ whose circular chromatic number is 4, and hence does not map to $\widehat{K}^s_{10;3}$. 
	\end{proof}
	
	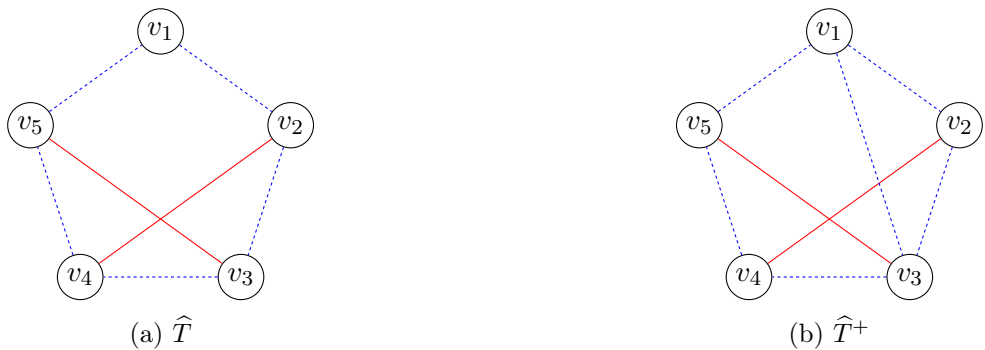
\begin{figure}[h!]
		\centering
		\begin{subfigure}[t]{0.45\textwidth}
			\centering
			\begin{tikzpicture}[rotate=90,scale=0.6, baseline=(current bounding box.center),
  node style/.style={circle,draw,inner sep=0.3pt, minimum size=6mm}, 
  dashed edge/.style={blue, dashed, thick},
  edge/.style={red, thick}
  ]
    \def\n{5}
    \def\radius{3}
    
    \foreach \i in {1,...,\n} {
      \node[node style] (P\i) at ({-360/\n * (\i - 1)}:\radius) {$v_{\i}$};
    }

    \foreach \i in {1,...,\n} {
      \pgfmathtruncatemacro\next{mod(\i, \n) + 1}
      \draw[blue, dash pattern=on .41mm off .41mm] (P\i) -- (P\next);
    }

    \draw[red] (P4) -- (P2);
    \draw[red] (P3) -- (P5);

\end{tikzpicture}
			\caption{$\widehat{T}$}
		\end{subfigure}
		\hfill
		\begin{subfigure}[t]{0.45\textwidth}
			\centering
		\begin{tikzpicture}[rotate=90,scale=0.6, baseline=(current bounding box.center),
  node style/.style={circle,draw,inner sep=0.3pt, minimum size=6mm}, 
  dashed edge/.style={blue, dashed, thick},
  edge/.style={red, thick}
  ]
    \def\n{5}
    \def\radius{3}
    
    \foreach \i in {1,...,\n} {
      \node[node style] (P\i) at ({-360/\n * (\i - 1)}:\radius) {$v_{\i}$};
    }

    \foreach \i in {1,...,\n} {
      \pgfmathtruncatemacro\next{mod(\i, \n) + 1}
      \draw[blue, dash pattern=on .41mm off .41mm] (P\i) -- (P\next);
    }

    \draw[red] (P4) -- (P2);
    \draw[red] (P3) -- (P5);
   \draw[blue,dash pattern=on .41mm off .41mm] (P1) -- (P3);

\end{tikzpicture}			
			\caption{$\widehat{T}^+$}
		\end{subfigure}
		
		\caption{The signed graphs $\widehat{T}$ and $\widehat{T}^+$}
		\label{fig:TandT+}
	\end{figure}

	Next we show that $\widehat{T}$ is minimal in the following sense.
	
	\begin{lemma}\label{lem:3-colorable}
		Assume $\widehat{G}$ is a signed simple graph such that $|V(\widehat{G})|\leq 5$,  $|E(\widehat{G})|\leq 7$, $\widehat{G}\ncong\widehat{T}$, and $\widehat{G}$ contains no subgraph isomorphic to $(K_4,-)$. Then $\widehat{G}$ is circular 3-colorable. 
	\end{lemma}
	
	\begin{proof}
		Let $\widehat{G}$ be a counterexample  to the claim having the minimum possible number of vertices.  We claim that every edge of $G$ is contained in a negative triangle. Suppose not, let $e$ be an edge which is in no negative triangle. By switching at one end of it, if needed, we may assume $e$ is a positive edge. Let $\widehat{G}'$ be the signed graph obtained from identifying the two ends of $e$ and removing positive loop and parallel edges of the same sign. Then $\widehat{G}'$ is not $10/3$-colorable either. By the minimality of $\widehat{G}$, $\widehat{G}'$ must be isomorphic to $(K_{4},-)$. But then $\widehat{G}\cong \widehat{T}$. Similarly, each pair of nonadjacent vertices must be in a negative 4-cycle.  Suppose not, let $u$ and $v$ be two nonadjacent vertices which are in no negative 4-cycle. By switching, we may assume that all edges from $u$ and $v$ to any common neighbor have the same sign. Let $\widehat{G}'$ be the signed graph obtained from identifying $v$, $u$  and removing parallel edges of the same sign. Then $\widehat{G}' \cong (K_{4},-)$. Note that every edge should be contained in a negative triangle. So $\widehat{G}$ consists of three consecutive negative triangles, each sharing an edge with the next  via a distinct shared edge. But then $\widehat{G}$ is 3-colorable.

		These condition implies that if $\widehat{G}$ has at most four vertices, then it is either $(K_3, -)$ which is 3-colorable or $(K_4,-)$ which is not supposed to be. 
		
		Thus we assume $V(G)=5$. Since $E(G)\leq 7$, and that $\delta(G)\geq 2$, there is a vertex $u$ of degree 2. Let $x$ and $y$ be the two neighbors of $u$. If $x\nsim y$, assuming $ux$ is a positive edge, by contracting it we get a graph on four vertices which is not 3-colorable. That is to say $\widehat{G}$ is a subdivision of $(K_4,-)$, i.e., $\widehat{G} \cong \widehat{T}$. Otherwise $xy\in E(G)$, but $u$, not being adjacent to the other two vertices $z$ and $t$, must be on a negative 4-cycle with each of them. That is to say each of $z$ and $t$ is adjacent to both $x$ and $y$. But then, noting that $|E(G)|\leq 7$, all edges are incident to either $x$ or $y$. Thus two among $u, z, t$, each together with $xy$ form a triangle of the same sign. These two vertices, after a switching if needed, can be identified to get a smaller homomorphic image of $\widehat{G}$.     
	\end{proof}

	We have observed so far that $\widehat{T}$, in some sense, uniquely maps to $\widehat{K}^s_{10;3}$, that is an embedding of it as a subgraph of $\widehat{K}^s_{10;3}$. This can be used to build $10/3$-critical signed graphs. Viewing  $\widehat{T}$ as a subgraph of $\widehat{K}^s_{10;3}$, if we add one of the missing edges with an opposite sign, then we will have a signed graph which is not $10/3$-colorable. Depending on which edge is added, the critical subgraph then is either switching equivalent to $(K_4,-)$, to the signed graph $\widehat{T}^+$ of \Cref{fig:TandT+}(b), or to the digon $\widetilde{K}_2$.

	\begin{lemma}\label{lem:Few_Vertex_Critical}
		Up to switching equivalence, the only $10/3$-critical signed graphs on at most five vertices are $(K_4,-)$ and $\widehat{T}^+$ and $\widetilde{K}_2$. 
	\end{lemma}
	
	\begin{proof}
		
		Let $\widehat{G}$ be a $10/3$-critical signed graph on at most 5 vertices. If $\widehat{G}$ is not a simple graph, then $\widehat{G}$ should be $\widetilde{K}_2$. Assume $\widehat{G}$ is simple. It follows from \Cref{lem:3-colorable}  that $\widehat{G}$ has at least four vertices and that if exactly four vertices, then it is switching equivalent to $(K_4,-)$. So we assume $|V(G)|=5$ and consider two cases: 
		
		{\bf Case 1} An edge $e$ of $\widehat{G}$ does not belong to a negative triangle. In this case we may switch one end of $e$, if needed, to ensure that $e$ is a positive edge. Then by contracting $e$ we have a homomorphic image which is a simple signed graph on four vertices and not $10/3$-colorable. Thus the image is switching equivalent to $(K_4,-)$, and hence $\widehat{G}$ contains a switching of $\widehat{T}$ as a subgraph.  We have then three possible edges to add to this subgraph $\widehat{T}$, namely: $v_{1}v_{3}$, $v_{1}v_{4}$ and $v_{2}v_{5}$. Adding $v_{2}v_{5}$ with a positive sign would result in a switching copy of $(K_4,-)$. Adding either of $v_{1}v_{3}$ or $v_{1}v_{4}$ with a positive sign would imply copies of $\widehat{T}^+$. Thus adding one of the three with a  positive sign is enough to reach $10/3$-critical subgraph. If none of them are added with a positive sign, then we will have a subgraph of $\widehat{K}^s_{10;3}$ and hence they are all $10/3$-colorable. 
		
		{\bf Case 2} Every edge of $\widehat{G}$ is in a negative triangle. Since we have at least eight edges and at most five vertices, at least two negative triangles share an edge $xy$. By switching, if needed, we may assume $xy$ is negative and that the other edges of the two triangles are positive. Let $u$ and $v$ be the other two vertices of the two triangles. If they are not adjacent, then they must be in a negative 4-cycle. This can only be through the last vertex in which case we must have a switching copy of $\widehat{T}$. The rest of the argument is the same as the previous case. If $uv$ is a negative edge, then the four vertices $x$, $y$, $u$, and $v$ induce a switching equivalent copy of $(K_4, -)$. If $uv$ is a positive edge, then it can only form a negative triangle with the fifth vertex in which case again we find a copy of $\widehat{T}$ as a subgraph of $\widehat{G}$ and the argument of the previous case applies.  	
	\end{proof}

	\begin{figure}[h!]
		\centering
		\begin{minipage}[t]{0.32\textwidth}
			\centering
			\begin{tikzpicture}[scale=0.8, baseline=(current bounding box.center),
    node style/.style={
        circle, draw,
        inner sep=0pt,
        minimum size=5mm,
        text height=1.2ex, text depth=.2ex 
    },
    dashed edge/.style={blue, thick, dash pattern=on .41mm off .41mm},
    edge/.style={red, thick}
]

				\node[node style] (w') at (2,3) { $w'$};
				\node[right=-1pt of w'] {\footnotesize -4};
				\node[node style] (v) at (3,1.5) {$v$};
				\node[right=-1pt of v] {\footnotesize 2};
				\node[node style] (u') at (4,0) {$u'$};
				\node[right=-1pt of u'] {\footnotesize -5};
				\node[node style] (w) at (2,0) {$w$};
				\node[below=-1pt of w] {\footnotesize 3};
				\node[node style] (v') at (0,0) { $v'$};
				\node[left=-1pt of v'] {\footnotesize 4};
				\node[node style] (u) at (1,1.5) {$u$};
				\node[left=-1pt of u] {\footnotesize 1};
				
				\draw[edge] (w') -- (v);
				\draw[edge] (u') -- (w);
				\draw[edge] (v') -- (u);
				
				\draw[dashed edge] (w') -- (u);
				\draw[dashed edge] (v) -- (u);
				\draw[dashed edge] (v) -- (w);
				\draw[dashed edge] (w) -- (u);
				\draw[dashed edge] (u') -- (v);
				\draw[dashed edge] (v') -- (w);
			\end{tikzpicture}
			\caption*{$\widehat{H}_1$}
		\end{minipage}
		\hfill
		\begin{minipage}[t]{0.32\textwidth}
			\centering

			\begin{tikzpicture}[scale=0.8, baseline=(current bounding box.center),
    node style/.style={
        circle, draw,
        inner sep=0pt,
        minimum size=5mm,
        text height=1.2ex, text depth=.2ex 
    },
    dashed edge/.style={blue, thick, dash pattern=on .41mm off .41mm},
    edge/.style={red, thick}
]
					
				\node[node style] (w') at (2,3) {$w^{\prime}$};
				\node[right=-1pt of w'] {\footnotesize 5}; 
				\node[node style] (v') at (0,0) {$v'$};
				\node[left=-1pt of v'] {\footnotesize -3}; 
				\node[node style] (u') at (4,0) {$u'$};
				\node[right=-1pt of u'] {\footnotesize 1}; 
				\node[node style] (w) at (2, 2) {$w$};
				\node[below=-1pt of w] {\footnotesize 1};
				\node[node style] (v) at (1,0.5) {$v$};
				\node[left=-1pt of v] {\footnotesize 4};
				\node[node style] (u) at (3,0.5) {$u$};
				\node[right=-1pt of u] {\footnotesize -2};
				
				\draw[edge] (u) -- (v)--(w)--(u);
				\draw[edge] (w') -- (u')--(v')--(w');
				\draw[edge] (w') -- (w);
				\draw[edge] (v) -- (v');
				\draw[edge] (u') -- (u);
				
			\end{tikzpicture}
			\caption*{$\widehat{H}_2$}
		\end{minipage}
		\hfill
		\begin{minipage}[t]{0.32\textwidth}
			\centering
			\begin{tikzpicture}[scale=0.8, baseline=(current bounding box.center),
    node style/.style={
        circle, draw,
        inner sep=0pt,
        minimum size=5mm,
        text height=1.2ex, text depth=.2ex 
    },
    dashed edge/.style={blue, thick, dash pattern=on .41mm off .41mm},
    edge/.style={red, thick}
]
				
				\node[node style] (w') at (2,3) { $w'$};
				\node[right=-1pt of w'] {\footnotesize -3};
				\node[node style] (v') at (0,0) { $v'$};
				\node[left=-1pt of v'] {\footnotesize -2}; 
				\node[node style] (u') at (4,0) { $u'$};
				\node[right=-1pt of u'] {\footnotesize -5}; 
				\node[node style] (w) at (2, 2) {$w$};
				\node[below=-1pt of w] {\footnotesize 1};
				\node[node style] (v) at (1,0.5) {$v$};
				\node[left=-1pt of v] {\footnotesize 4};
				\node[node style] (u) at (3,0.5) {$u$};
				\node[right=-1pt of u] {\footnotesize -2};
				
				\draw[edge] (u) -- (v)--(w)--(u);
				\draw[edge] (u')--(v');
				\draw[edge] (w') -- (w); 
				\draw[dashed edge] (v')--(w')--(u');
				\draw[edge] (v) -- (v');
				\draw[edge] (u') -- (u);
				
			\end{tikzpicture}        
			\caption*{$\widehat{H'}_2$}
		\end{minipage}
		\\
		\vspace{5mm}
		\begin{minipage}[t]{0.32\textwidth}
			\centering
			\begin{tikzpicture}[scale=0.8, baseline=(current bounding box.center),
    node style/.style={
        circle, draw,
        inner sep=0pt,
        minimum size=5mm,
        text height=1.2ex, text depth=.2ex 
    },
    dashed edge/.style={blue, thick, dash pattern=on .41mm off .41mm},
    edge/.style={red, thick}
]
				
				\node[node style] (w) at (5, 2.5) {$w$};
				\node[right=-1pt of w] {\footnotesize -5};
				\node[node style] (x) at (1.6,0.75) {};
				\node[left=-1pt of x] {\footnotesize -2};
				\node[node style] (u) at (3.1,3) {$u$};
				\node[right=-1pt of u] {\footnotesize -5};
				\node[node style] (y) at (3.1,1.5){};
				\node[above=-1pt of y] {\footnotesize 1};
				\node[node style] (z) at (3.1,0) {};
				\node[below=-1pt of z] {\footnotesize 4};
				\node[node style] (v) at (4.6,0.75) {$v$};
				\node[right=-1pt of v] {\footnotesize -3};
				
				\draw[edge] (u) -- (x)--(y)--(z)--(v)--(y);
				\draw[edge] (x) -- (z);
				
				\draw[dashed edge] (w) -- (u);
				\draw[dashed edge] (w) -- (v);
				\draw[dashed edge] (u) -- (v);
			\end{tikzpicture}
			\caption*{$\widehat{H}_3$}
		\end{minipage}
		\hfill
		\begin{minipage}[t]{0.32\textwidth}
			\centering
			\begin{tikzpicture}[scale=0.8, baseline=(current bounding box.center),
    node style/.style={
        circle, draw,
        inner sep=0pt,
        minimum size=5mm,
        text height=1.2ex, text depth=.2ex 
    },
    dashed edge/.style={blue, thick, dash pattern=on .41mm off .41mm},
    edge/.style={red, thick}
]
				\node[node style] (v) at (0.1, 0.75) {$v$};
				\node[below=-1pt of v] {\footnotesize 2};
				\node[node style] (x) at (1.6,0.75) {$x$};
				\node[left=-1pt of x] {\footnotesize -2};
				\node[node style] (u) at (3.1,3) {$u$};
				\node[right=-1pt of u] {\footnotesize -5};
				\node[node style] (w2) at (3.1,1.5) {$w_2$};
				\node[above=-1pt of w2] {\footnotesize 1};
				\node[node style] (z) at (3.1,0) {$z$};
				\node[below=-1pt of z] {\footnotesize 4};
				\node[node style] (y) at (4.6,0.75) {$y$};
				\node[below=-1pt of y] {\footnotesize -3};
				
				\draw[edge] (u) -- (x)--(w2)--(z)--(y)--(w2);
				\draw[edge] (x) -- (z);
				
				\draw[dashed edge] (v) -- (u);
				\draw[dashed edge] (v) -- (z);
				\draw[dashed edge] (u) -- (y);
			\end{tikzpicture}
			\caption*{$\widehat{H}_4$}
		\end{minipage}
		\hfill
		\begin{minipage}[t]{0.32\textwidth}
			\centering
			\begin{tikzpicture}[scale=0.8, baseline=(current bounding box.center),
    node style/.style={
        circle, draw,
        inner sep=0pt,
        minimum size=5mm,
        text height=1.2ex, text depth=.2ex 
    },
    dashed edge/.style={blue, thick, dash pattern=on .41mm off .41mm},
    edge/.style={red, thick}
]				\node[node style] (v) at (0.1, 0.75) {$v$};
				\node[below=-1pt of v] {\footnotesize -1};
				\node[node style] (x) at (1.6,0.75) {$x$};
				\node[left=-1pt of x] {\footnotesize -2};
				\node[node style] (u) at (3.1,3) {$u$};
				\node[right=-1pt of u] {\footnotesize -5};
				\node[node style] (w2) at (3.1,1.5) {$w_2$};
				\node[above=-1pt of w2] {\footnotesize 1};
				\node[node style] (z) at (3.1,0) {$z$};
				\node[below=-1pt of z] {\footnotesize 4};
				\node[node style] (y) at (4.6,0.75) {$y$};
				\node[below=-1pt of y] {\footnotesize -3};
				
				\draw[edge] (u) -- (x)--(w2)--(z)--(y)--(w2);
				\draw[edge] (x) -- (z);
				
				\draw[edge] (v) -- (u);
				\draw[dashed edge] (v) -- (z);
				\draw[dashed edge] (u) -- (y);
			\end{tikzpicture}
			\caption*{$\widehat{H}'_4$}
		\end{minipage}
		\caption{10/3-colorable graphs on six vertices}
		\label{fig:10/3-colorable graphs}
	\end{figure}
	\begin{lemma} \label{lem:10/3-colorable_graphs}
		Each of the signed graphs, $\widehat{H}_1, \widehat{H}_2, \widehat{H}'_2, \widehat{H}_3, \widehat{H}_4, \widehat{H}'_4$, depicted in \Cref{fig:10/3-colorable graphs}, admits a 10/3-coloring.
	\end{lemma}
	\begin{proof}
		For each of them, a 10/3-coloring is given in \Cref{fig:10/3-colorable graphs}.
	\end{proof}
	
	\subsection{List $10/3$-coloring}
	
	List assignment is a natural consequence of inductive approach. Having colored a few vertices, each remaining vertex $v$ has a list $L(v)$ of available colors and a list $\bar{L}(v)$ of forbidden colors. The goal is then to choose colors from the list $L(v)$ for each vertex $v$ that is not colored such that the conditions for all the edges are satisfied. To this end, we prefer to work with edge-sign preserving mappings to $K_{10;3}^s$ rather than homomorphism to $\widehat{K}_{10;3}^s$. Thus $\pm[5]=\{ \pm 1, \pm 2, \pm 3, \pm 4, \pm 5\}$ is the set of all colors and a coloring $\phi$ must satisfy the condition that if $x\sim y$, then $\phi(x)$ is adjacent to $\phi(y)$ in $K_{10;3}^s$ with an edge of the sign $\sigma(xy)$.
	
	In most cases, rather than working with the actual list $\mathcal{L}=\{L(v)\}$ of available colors for a set of vertices, we will work with the number $f(v)=|L(v)|$ of colors available at each vertex. Our aim then would be to show that if $f(v)$ is large enough for each $v$, then regardless of the structure of $L(v)$ there will always be a coloring. In such cases, the signed graph under consideration will be referred to as $f$-list-colorable. In some cases, we will also need the structure of $L(v)$ in $K_{10;3}^s$ to achieve our goal. 
		
	A warm up observation is the following.
	
	\begin{observation}\label{obs:counting_partial_coloring-K2}
		Given $(K_2, \sigma)$, on vertices $x$ and $y$, if $x$ is already colored, then $|L(y)|=5$ and $|\bar{L}(y)|=5$. Moreover, $L(y)$ is a list of five consecutive colors in the cyclic ordering of $\pm [5]$. 
	\end{observation}
	
	This is based on the fact that the set of positive neighbors of each vertex $v$ of $K_{10;3}^s$ is the consecutive interval of length five with $v$ at its center. The set of negative neighbors is the complement, hence also a consecutive interval of length five.

	Before stating the next lemma, we note that switching a vertex $x$ and changing its color to $-\phi(x)$ has no impact at the list of available colors at its neighbors. Similarly, after a switching at a vertex $u$, the list of available colors changes from $L(u)$ to $-L(u)$. Thus, in many arguments, we may select a desired equivalent signature, and assume colors and lists are modified to fit the chosen signature.

	\begin{lemma}\label{lem:counting_partial_coloringK3}
		Given a positive triangle $uvw$, if $u$ and $w$ are already colored, then $|L(v)|$ is a list of  three or four or five consecutive colors.
	\end{lemma}

	\begin{proof}
		We assume all three edges are positive. Assume $\phi$ is a proper coloring of $uw$. Then, up to symmetries, we have three options for the pair $(\phi(u), \phi(w))$:
		$(\phi(u), \phi(w))=(1,1)$, $(\phi(u), \phi(w))=(1,2)$, and $(\phi(u), \phi(w))=(1,3)$. In the first case we have $L(v)=\{-4,-5,1,2,3\}$, in the second case $L(v)=\{-5,1,2,3\}$ and in the last case $L(v)=\{1,2,3\}$.
	\end{proof}
	
	The next lemma is easy to observe from \Cref{fig:circular_clique}.
	
	\begin{lemma}\label{lem:distance_i} Given two vertices $x, y$ of $K_{10;3}^s$, if their cyclic distance is $i$ ($1\leq i\leq 5$), then the number of vertices joined to at least one of them by a positive edge (similarly by a negative edges), is $5+i$. 
	\end{lemma}	
	
	Noting that in any set $X$ of vertices we have  two elements of circular distance at least $\min\{5, |X|-1\}$, we conclude that:
	
	\begin{lemma}\label{lem:a+4} For any set  $X$ of colors, the number of positive (or negative) neighbors of $X$ is at least $\min \{10, |X|+4\}$.
	\end{lemma}

	\begin{lemma}\label{lem:K2P3Coloring}    
		In each of the following cases, the given signed graph together with the associated assignment $f$ is $f$-list-colorable.
		\begin{itemize}	
			\item[(a)] $(K_2, \sigma)$ with $u,v$  as vertices and $f$ satisfying $f(u), f(v)\geq 1$ and $f(u)+f(v)\geq 7$.
			
			\item[(b)] $(P_3, \sigma)$ on vertices $v_1,v_2,v_3$ in the order of the path and $f$ satisfying: $f(v_i)\geq 1$ for $i=1,2,3$, $f(v_i)+f(v_{i+1})\geq 7$ for $i=1,2$ and $f(v_1)+f(v_2)+f(v_3)\geq 13$.  
		\end{itemize}
	\end{lemma}

	\begin{proof}		
		For $(a)$, assume $\mathcal{L}$ is a list assignment with $|L(u)|=f(u)=a$ ($a\leq 6$) and $|L(v)|=f(v)\ge 7-a$. It follows from \Cref{lem:a+4} that there exists a color set $C(v)$ with $|C(v)|=a+4$ such that for each color in $C(v)$, there is a suitable color $\phi(u)\in L(u)$. Since $7-a+a+4>10$, $C(v)\cap L(v)\neq \emptyset$, hence we have suitable colors for both vertices from their lists of available colors.

		For $(b)$, assume $\mathcal{L}$ is a list assignment with $|L(v_1)|=f(v_1), |L(v_2)|=f(v_2) \text{ and } |L(v_3)|=f(v_3).$ Again, applying \Cref{lem:a+4} on $v_1v_2$ we have a set $C(v_2)$ of at least $\min\{10, f(v_1)+4\}$ colors for each of which there is a suitable color in $L(v_1)$. Basic inclusion exclusion implies $|C(v_2)\cap L(v_2)|\ge \min\{10, f(v_1)+4+f(v_2)-10\}=\min\{10, f(v_1)+f(v_2)-6\}$. 
		
		Applying \Cref{lem:K2P3Coloring}(a) to $(v_2v_3)$ with $C(v_2)\cap L(v_2)$ as the list of $v_2$ and $L(v_3)$ as the list of $v_3$, noting that $|C(v_2)\cap L(v_2)|+|L(v_3)|\geq f(v_1)+f(v_2)-6+f(v_3)\geq 7$, we have a valid assignment of colors $c(v_2)\in C(v_2)\cap L(v_2)$ and $c(v_3)\in L(v_3)$ after which, by the choice of $C(v_2)$, we may select a valid color $c(v_1)\in L(v_1)$. 
	\end{proof}

	\begin{lemma}\label{lem:4-cycle}
		Let $(C_4, \sigma)$ be the signed graph where $v_{1}, v_{2}, v_{3}, v_{4}$ are the vertices of $C_4$ in cyclic order and $v_{3}v_{4}$ is the only positive edge. Let $\mathcal{L}$ be a list assignment where $L(v_{1})$ and $L(v_{2})$ are intervals of length 7 and  $L(v_{3})$ and $L(v_{4})$ are intervals of length 5. Then $(C_4, \sigma)$ is $\mathcal{L}$-colorable.
	\end{lemma}

	\begin{proof}
		Based on $|L(v_{3})\cap L(v_{4})|$, we consider three cases:\\
		{\bf Case $|L(v_{3})\cap L(v_{4})|\ge 2$.}
		Since the lists are intervals of colors, we can find a color  $c\in L(v_{4})$ such that $|N^{+}_{K_{10;3}^{s}}(c)\cap L(v_{3})|\ge 4$.
		Let $L'(v_{1})=L(v_{1})\cap N^{-}_{K_{10;3}^{s}}(c)$, $L'(v_{2})=L(v_{2})$, and $L'(v_{3})=L(v_{3})\cap N^{+}_{K_{10;3}^{s}}(c)$. Observe that  $|L'(v_{1})|\ge 2$, $|L'(v_{2})|=7$, and $|L'(v_{3})|\geq 4$.  Then after assigning color $c$ to $v_4$, by \Cref{lem:K2P3Coloring}(b), the path $v_{1}v_{2}v_{3}$ admits an $\mathcal{L'}$-coloring and thus we have an $\mathcal{L}$-coloring of $(C_4, \sigma)$.   
		
		{\bf Case $|L(v_{3})\cap L(v_{4})|=1$}. Without loss of generality, let $L(v_{4})=\{1,2,3,4,5\}$ and $L(v_{3})=\{5,-1,-2,-3,-4\}$.  Consider the coloring $\phi(v_{4})=5$. In the modified list for $v_{3}$, we have three colors: $5$, $-1$, $-2$. If $|L(v_{1})\cap N^{-}_{K_{10;3}^{s}}(5)|\ge 3$, then we are done by applying \Cref{lem:K2P3Coloring}(b) on the path $v_1v_2v_3$ with modified lists. Otherwise we have $\{3,4,5,-1,-2\} \subseteq L(v_{1})$. Repeating the same argument starting with assignment of $5$ to $v_3$ and applying \Cref{lem:K2P3Coloring}(b) on the path $v_2v_1v_4$ we conclude that $\{3,4,5,-1,-2\} \subseteq L(v_{2})$. Thus an $\mathcal{L}$-coloring of $(C_4, \sigma)$ is as follows: $\phi(v_{1})=-1$, $\phi(v_{2})=3$, $\phi(v_{3})=-4$ and $\phi(v_{4})=1$.

		{\bf Case $L(v_{3})\cap L(v_{4})=\emptyset$}. Without loss of generality, let $L(v_{4})=\{1,2,3,4,5\}$ and $L(v_{3})=\{-1,-2,-3,-4,-5\}$.  We first consider $\phi(v_{4})=1$,  in the modified list for $v_{3}$, we have two colors: $-4$, $-5$. If $|L(v_{1})\cap N^{-}_{K_{10;3}^{s}}(1)|\ge 4$, then we are done by \Cref{lem:K2P3Coloring}(b).  Thus we assume $|L(v_{1})\cap N^{-}_{K_{10;3}^{s}}(1)|\leq 3$. By symmetry, $|L(v_{1})\cap N^{-}_{K_{10;3}^{s}}(5)|\leq 3$. Hence $L(v_{1})=\{-5,1,2,3,4,5,-1\}$. By a similar argument we conclude that $L(v_{2})=-L(v_{1})$. Then an $\mathcal{L}$-coloring $\phi$  is as follows: $\phi(v_{1})=4$, $\phi(v_{2})=-2$, $\phi(v_{3})=-5$ and $\phi(v_{4})=1$.
	\end{proof}

	\begin{lemma}\label{lem:K_{2,3}}
		Given $(K_{2,3}, \sigma)$, let $\mathcal{L}$ be a list assignment such that each vertex of the smaller part of $K_{2,3}$, $\{u,v\}$, has a full list (size 10), and each of the vertices of the other part, $\{x_{1}, x_{2}, x_{3}\}$, is assigned an interval of length 5. Then  $(K_{2,3}, \sigma)$ is $\mathcal{L}$-colorable. 
	\end{lemma}
	
	\begin{proof}
		As $(K_{2,3}, \sigma)$ contains a positive 4-cycle for any choice of $\sigma$, we may assume $ux_{1}vx_{2}$ is a positive 4-cycle, and, if needed, after a suitable switching together with modification of $\mathcal{L}$, we may assume all edges of $ux_{1}vx_{2}$ are negative. Since $L(x_{1})$ and $L(x_{2})$ are intervals of length 5, one can choose a pair $\phi(x_{1}), \phi(x_{2})$ of colors with cyclic distance 0 or 1. It follows from \Cref{lem:distance_i}  that $|N^{-}_{K_{10;3}^{s}}(\phi(x_{1}))\cap N^{-}_{K_{10;3}^{s}}(\phi(x_{2}))| \geq 4$. After modifying the lists of $u$ and $v$, it is ensured by \Cref{lem:K2P3Coloring}(b) that the coloring can be extended to the remaining three vertices.
	\end{proof}
	

	\begin{lemma}\label{lem:NegativeTirnagle}
		Let $u,v,w$ be the three vertices of $\widehat{G}=(K_3,-)$ and $\mathcal{L}$ be a list-assignment where $|L(u)|+|L(v)|+|L(w)|=18$. Then  $\widehat{G}$ is $\mathcal{L}$-colorable unless we have one of the following cases: 
		
		\begin{itemize}
			\item[(1)] One of the lists is empty.
			
			\item[(2)] $L(u)=L(v)=L(w)=X$, where the subgraph induced by the negative edges of $K_{10;3}^{s}[X]$ is bipartite.
			
			\item[(3)] Up to an isomorphism, $L(u)=L(v)=\{3,4,5,-1,-2,-3,-4\}$, and $L(w)=\{4,5,-2,-3\}$.
			
			\item[(4)] Up to an isomorphism, $L(u)=L(v)=\{3,4,5,-1,-2,-3,-4. -5\}$, and $L(w)=\{5,-3\}$.
			
		\end{itemize}	
		
	\end{lemma}
	
	\begin{proof}
		Consider a $(K_3, -)$ in $K_{10;3}^{s}$, for example $1,4, -3$.  Mapping $\widehat{G}$ to this triangle is the same as choosing a matching in the bipartite graph formed by $\{1, 4, -3\}$ on the one part and $\{L(u), L(v), L(w)\}$ on the other part where membership determines adjacencies. Thus by Hall's Theorem, if $\widehat{G}$ cannot be mapped to this triangle, then it is for one of the three possible reasons:
		
		\begin{itemize}
			\item[(I)] One of $L(u)$, $L(v)$, or $L(w)$ does not contain any of $1$, $4$, and $-3$.
			\item[(II)] Two of $L(u)$, $L(v)$ and $L(w)$ contain only one of $1$, $4$, and $-3$.
			\item[(III)] $L(u)$, $L(v)$ and $L(w)$ together contain only two of $1$, $4$, and $-3$, that is to say one color appears in none of the three lists. Such a color is referred to as type $m$.
		\end{itemize} 
	
		Consider the pairs $(i,x)$ where $i$ is a color (a vertex of $K_{10;3}^{s}$) and $x$ is one of $u$, $v$ or $w$. Thus we have a total of 30 pairs, of which 18 represents colors in $L(u)$, $L(v)$ and $L(w)$, and hence 12 represent a color forbidden on a vertex. If $\widehat{G}$ cann't map to the triangle $143^{-}$ because of (I) or (III), then we have three pairs representing forbidden colors, and if $\widehat{G}$ cann't map to the triangle $143^{-}$ because of (II), then we have four such pairs. 
		
		Assume both of  $1$ and $1^-$ are type $m$. It means six pairs are forbidden: $(1,u)$, $(1,v)$ and $(1, w)$; $(-1,u)$, $(-1,v)$ and $(-1, w)$. Considering the disjoint triangles $253^-$ and $32^{-}5^{-}$, each of them contributes at least 3 forbidden pairs to the list of forbidden pairs. Since there are 12 forbidden pairs in total, the colors $4$ and $-4$ must be in every list. Similarly, considering the disjoint triangles $254^-$ and $42^{-}5^{-}$, the colors $3$ and $-3$ must be in every list. 
		Note that  each triangle ($253^-$ and $32^{-}5^{-}$, $254^-$ and $42^{-}5^{-}$) contributes exactly 3 forbidden pairs as there are 12 forbidden pairs in total. Thus in particular, item (II) will not happen for any of these triangles. 
		For the negative triangle $253^{-}$, item (I)  will not apply because $-3$ is in every list. Thus one of $2$ and $5$ must be in every list, and the other in none. The same holds for $-2$ and $-5$ when considering the negative triangle $42^{-}5^{-}$. For each color, it is either contained in all lists  $L(u)$, $L(v)$, $L(w)$, or in none of them. Thus $L(u)=L(v)=L(w)$.  There are  four cases for the list: $\{2, 3, 4, -2, -3, -4\}$, $\{2, 3, 4, -3, -4, -5\}$, $\{3, 4, 5, -2, -3, -4\}$, $\{3, 4, 5, -3, -4, -5\}$ . All cases fall in the condition (2) of the lemma.

		Assume both of $1$ and $5$ are type $m$. Then repeating similar argument, by considering disjoint triangles $31^{-}4^{-}$ and $42^{-}5^{-}$,  we conclude that $2$ and $-3$ are in every list. Similarly, working with disjoint triangles $21^{-}4^{-}$ and $32^{-}5^{-}$,  we conclude that $4$ and $-3$ are contained in every list, and disjoint triangles $21^{-}4^{-}$ and $42^{-}5^{-}$ imply $3$ and $-3$ are contained in every list.  It still holds that each triangle ($21^{-}4^-$ and $32^{-}5^{-}$, $31^{-}4^-$ and $42^{-}5^{-}$) contributes exactly 3 forbidden pairs. 
		Considering the triangle $32^{-}5^{-}$, since it contributes exactly 3 forbidden pairs and $3$ is in every list, one of $-2$ and $-5$ must be in every list, and the other in none. Similarly, the negative triangle $21^{-}4^{-}$ implies one of $-1$ and $-4$ is in every list, the other is none.  Thus $L(u)=L(v)=L(w)$. It is not hard to check all cases fall in the condition (2) of the lemma.   
		
		Assume both of $1$ and $4$ are type $m$. The colors $1$ and $4$ have six forbidden pairs.  Since for the triangles $254^-$ and $31^-5^-$, each is with exactly three forbidden colors, we conclude that $-3$ and $-2$ is in all the three lists.  Then considering the triangles  $253^-$ and $31^-4^-$, $-2$ and $-5$ should be in every list. For the triangles  $21^{-}4^-$ and $32^-5^-$, $-3$ and $5$ should be in every list. Considering triangles $253^-$, $32^-5^-$,  since $5, -2, -3, -5$ are all in every list, we conclude that $2$ and $3$ are in no list which means $L(u)=L(v)=L(w)=\{5,-1,-2,-3,-4,-5\}$ and hence the list assignment falls in condition (2) of the lemma.

		If $1$ and $3$ are both of type $m$, then, repeating a similar argument, we first conclude that $5$, $-1$, $-3$, and $-4$ are in every list. In the triangle $253^{-}$, as $5$ and $-3$ are in every list,  $2$ is of type $m$ as well. Thus colors $1$, $2$ and $3$ are already in a total of at least nine forbidden pairs. For the triangle $42^-5^-$, there are two possibilities: Either exactly two of the vertices are in every list, but any pair of them together with  $5$, $-1$, $-3$ and $-4$ induces a bipartite subgraph of negative edges, which is condition (2) of the lemma. Or one of the lists contains none of $4, -2, -5$ and we have the condition (3) of the lemma. 
		
		Assume $1$ and $2$ are of type $m$. We may assume these two are the only type $m$ vertices as otherwise we have one of the previous cases. Then considering the triangles $31^{-}4^{-}$ and $42^{-}5^{-}$, $5$ and $-3$ are in all lists. Then each of the two triangles $31^-4^-$  and $42^-5^-$ must fall in item (I): one of $L(u)$, $L(v)$, $L(w)$ does not contain any color of the triangle. If $L(u)$ does not contain any  vertex of $31^-4^-$ and $42^-5^-$, then we have the last case (4) of the lemma. Otherwise, without loss of generality, we may assume $L(u)\cap \{3,-1,-4\}=\emptyset$ and $L(v)\cap \{4,-2,-5\}=\emptyset$. We assign $3$, $-2$, $-5$ to $v$, $u$, $w$ respectively.

		For the remaining cases, we assume at most one vertex is of type $m$. If there is such a vertex, then assume that is $1$, otherwise assume $1$ is a vertex that appears in only one list. For the three triangles, $253^-$, $31^-4^-$, and $42^-5^-$, at most 10 pairs are forbidden. Thus at least two of $253^-$, $31^-4^-$ and $42^-5^-$ satisfy item (I), that is to say one of $L(u)$, $L(v)$ and $L(w)$ does not contain any color of the chosen triangle. By the symmetry of $253^-$ and $42^-5^-$ with respect to $1$, we assume $\{2,5,-3\}\cap L(u)=\emptyset$. Depending on which is the other triangle and what color it is missing, we have several cases.
		
		{\bf Case 1} $\{4, -2, -5\}\cap L(u)=\emptyset$.
		First we have $\{2, 5, -3, 4, -2, -5\}\subseteq L(v), L(w)$. Since $L(u)\neq \emptyset$, one of $1, 3,-1,-4$ can be assigned to $u$. If $L(u)\cap \{1, 3,-4\}\neq \emptyset$, then the corresponding assignments for $(u,v,w)$ are $(1,4,-3)$,  $(3,-2,-5)$ and  $(-4,2,5)$. Then $L(u)\cap \{1, 3, -1, -4\} =\{-1\}$, and $(3,u)$, $(-4,u)$ are forbidden pairs. But then negative triangle $31^{-}4^{-}$ does not satisfy any of (I), (II), and (III).	
				
		{\bf Case 2} $\{3, -1, -4 \}\cap L(u)=\emptyset$. 
			Similarly, we have $\{2, 5, -3, 3, -1, -4\}\subseteq L(v), L(w)$. If $L(u)\cap \{1, -5\}\neq \emptyset$, then the corresponding assignments for $(u,v,w)$ are $(1,5,-3)$ and  $(-5,3,-1)$. Thus we may assume that $L(u)\cap \{1, -5\}= \emptyset$. Since $L(u)\neq \emptyset$, we have $L(u)\cap \{4,-2\}\neq \emptyset$. 
			
    Assume $4 \in L(u)$. If $1$ is not type $m$, then  $(4, 1, -3)$ or  $(4, -3,1)$ is a corresponding assignment for $(u,v,w)$. Hence we may assume that $1$ is type $m$. Thus there are three forbidden pairs in the negative triangle $42^{-}5^{-}$. However, each of $L(u), L(v), L(w)$ can miss at most two colors of the triangle, since $(-5,u)$ is a forbidden pair while $(4,u)$ is not. This implies that item (I) doesn't hold. Recall that there are three forbidden pairs in $42^{-}5^{-}$, and $1$ is the only one of type $m$. Hence none of (I), (II), (III) can be satisfied.
    
    Assume that $L(u)=-2$. Since $(-2, 3, -5)$, $(-2, -5, 3)$, $(-2, 4, -5)$ and $(-2, -5, 4)$ are possible assignments for $(u,v,w)$, it follows that $\{-2, -5\}\cap L(v)=\emptyset$ and $\{-2, -5\}\cap L(w)=\emptyset$. This contradicts the fact that there are 12 forbidden pairs in total.
					
		 {\bf Case 3} $\{4, -2, -5\}\cap L(v)=\emptyset$. We know that $\{2, 5, -3\}\subseteq L(v), L(w)$ and $\{4, -2, -5\}\subseteq L(u), L(w)$. If $1$ is not type $m$, then $(1, 5, -3)$, $(4,1, -3)$ and $(4, -3, 1)$ are the corresponding assignments for $(u,v,w)$. Thus $1$ is type $m$ and the negative triangle $31^{-}4^{-}$ should satisfy item (I). If  $\{3, -1, -4\}\cap L(w)=\emptyset$ or  $\{3, -1, -4\}\cap L(v)=\emptyset$, then  $(-4, 5, 2)$  is an  assignment for $(u,v,w)$. Otherwise $(4, -5, 2)$ is an assignment for $(u,v,w)$.
		 		
		{\bf Case 4} $\{3, -1, -4 \}\cap L(v)=\emptyset$.  We know that $\{2, 5, -3\}\subseteq L(v), L(w)$ and $\{3, -1, -4\}\subseteq L(u), L(w)$. In this case, $(-4, 2, 5)$ is an assignment for $(u, v, w)$ and we are done.		
			\end{proof}

	The following lemma is an easy exercise and we leave the proof to the reader.
	
	\begin{lemma}\label{lem:2-vertex}
		Let $x,v,y$ be vertices of $(P_3,\sigma)$, with $v$ being the degree 2 vertex. Then a 10/3-coloring $\phi$ of $x$ and $y$ extends to $v$ except in the following two cases:
		
		\begin{itemize}
			\item $\phi(x) = -\phi(y)$  and $\sigma(xv)=\sigma(vy)$,
			\item $\phi(x) = \phi(y)$ and $\sigma(xv)= -\sigma(vy)$.
			
		\end{itemize}
	\end{lemma}

	\section{Structural properties} 
In this section, we apply the potential technique to provide a list of forbidden configurations in a counterexample to our claim with minimum number of vertices. To this end we first restate our claim.

\begin{theorem}\label{thm:main}
	If $\widehat{G}$ is a 10/3-critical signed graph which is isomorphic neither to $(K_4,-)$ nor to $\widetilde{K}_2$, then $$p(\widehat{G})=3|V(\widehat{G})|-2|E(\widehat{G})|\leq -1.$$
\end{theorem}

Throughout this section, we assume $\widehat{G}$ is a counterexample with minimum number of vertices to \Cref{thm:main}. That is to say first of all, $\widehat{G}$ is $10/3$-critical, secondly, $p(\widehat{G})\geq 0$, and thirdly, for any $10/3$-critical signed graph $\widehat{H}$ with $|V(\widehat{H})|<|V(\widehat{G})|$, if $\widehat{H}$ is isomorphic neither to $(K_4,-)$ nor $\widetilde{K}_2$, then $p(\widehat{H})\leq -1$.

Given a signed graph $\widehat{H}$, $P_2(\widehat{H})$ is the class of signed graphs obtained from $\widehat{H}$ by adding a new vertex and joining it to two vertices of $\widehat{H}$ with arbitrary signs. The following is a key lemma of the proof where we show all proper subgraphs of $\widehat{G}$ are even sparser than $\widehat{G}$ itself.  

\begin{lemma}\label{lem:key}
	Assume $\widehat{H}\subseteq \widehat{G}$. Then 
	\begin{itemize}
		\item[(I)] $p(\widehat{H})\ge 0$, if $\widehat{H} \cong \widehat{G}$;
		\item[(II)] $p(\widehat{H})\ge 1$, if $\widehat{G}\in P_2(\widehat{H})$ or 
		$\widehat{H}\cong \widehat{T}$;
		\item[(III)] $p(\widehat{H})\ge 2$, otherwise.
	\end{itemize}
\end{lemma}

\begin{proof}
	The statement (I) is the assumption on $\widehat{G}$.
	If $\widehat{G}\in P_2(\widehat{H})$, then $p(\widehat{H})=p(\widehat{G})-(3\times 1-2\times 2)=p(\widehat{G})+1\ge 1$. The potential value of $\widehat{T}$ is $3\times 5-2\times 7=1$.

	Suppose (III) is not true. Let $\widehat{H}$ be a maximal proper subgraph of $\widehat{G}$ satisfying $\widehat{G} \not\in P_2(\widehat{H})$, $\widehat{H}\ncong \widehat{T}$, and $p(\widehat{H})\leq 1$. Since adding an edge decreases the value of the potential, we may assume $\widehat{H}$ is an induced subgraph of $\widehat{G}$.

	Observe that $p(\widehat{K}_1)=3$, $p(\widehat{K}_2)=4$, $p(\widehat{K}_3)=3$. Assume $\widehat{H}$ has $4$ or $5$ vertices. Since $p(\widehat{H})\leq 1$, we have $E(\widehat{H})\leq 7$ and hence  by \Cref{lem:3-colorable}, $\widehat{H}$ maps to a negative triangle. If it has more than $6$ vertices, then since it is a proper subgraph, it maps to $K_{10;3}^{s}$. In both cases, we conclude that there exists a homomorphism $\phi$ of $\widehat{H}$ to $\widehat{K}^s_{10;3}$ which identifies at least two vertices. In other words, $|\phi(\widehat{H})|<|V(\widehat{H})|$.
	We extend $\phi$ to $\widehat{G}$ by taking it as an identity mapping on the rest of $\widehat{G}$.

	Define $\widehat{G}^*$ to be the signed graph obtained from $\phi(\widehat{G})$ as follows: 
	\begin{itemize}
		\item if $\phi(\widehat{H})$ contains no subgraph isomorphic to $\widehat{T}$, then $\widehat{G}^*=\phi(\widehat{G})$;
		\item if $\widehat{T}_{1}\subseteq \phi(\widehat{H})$ and $\widehat{T}_1\cong \widehat{T}$, then $\widehat{G}^*=\phi(\widehat{G})-E_d$, where $E_d=E(\phi(\widehat{H}))-E(\widehat{T}_{1})$.
	\end{itemize}
	
	In other words, after identifying,  the $\widehat{H}$-part of $\widehat{G}$ with a subgraph of  $\widehat{K}^{s}_{10;3}$, if the image of $\widehat{H}$ contains more edges than $\widehat{T}$, then we delete those extra edges.

	By \Cref{lem:T}, if $\widehat{G}^*\to \widehat{K}^s_{10;3}$, then $\phi(\widehat{G})\to \widehat{K}^s_{10;3}$, which would imply that $\widehat{G}$ is $10/3$-colorable. Hence $\widehat{G}^*\not\to \widehat{K}^s_{10;3}$.

	Since  $\widehat{G}^*\not\to \widehat{K}^s_{10;3}$, $\widehat{G}^*$ contains a $10/3$-critical graph $\widehat{G}'$. As $|V(\widehat{G}^*)|<|V(\widehat{G})|$, we have $|V(\widehat{G}')|<|V(\widehat{G})|$.

	\begin{claim}\label{clm:digon}
		$\widehat{G}'\ncong \widetilde{K}_2$.
	\end{claim}
	
	\noindent\emph{Proof of the claim.} Assume to the contrary, and suppose vertices $u, v$ induce a digon in $\widehat{G}^*$. Since there is no digon in $\widehat{G}$ and $\widehat{K}^s_{10;3}$, one of $u$ and $v$, say $u$, is in $\phi(\widehat{H})$, the other one is in $V(\widehat{G})-V(\widehat{H})$. Moreover, $v$ is adjacent to at least two vertices of $\widehat{H}$. So $$p(\widehat{G}[V(\widehat{H})+v])\leq p(\widehat{H})+3\times 1-2 \times 2=p(\widehat{H})-1<p(\widehat{H}) .$$
	
	By the choice of $\widehat{H}$,  either $\widehat{G}[V(\widehat{H})+v]$ is isomorphic to $\widehat{G}$ or $\widehat{G}\in P_2(\widehat{G}[V(\widehat{H})+v])$. The latter is impossible because $p(\widehat{G}[V(\widehat{H})+v])\leq p(\widehat{H})-1\leq 0$. Hence $\widehat{G}[V(\widehat{H})+v]=\widehat{G}$. Moreover, $d(v)=2$, as otherwise $p(\widehat{G}[V(\widehat{H})+v])\leq p(\widehat{H})-3\leq -2$. That means $\widehat{G}\in P_2(\widehat{H})$, which contradicts the assumption.

	\begin{claim}\label{clm:K_4}
		$\widehat{G}'\ncong (K_4,-)$.
	\end{claim}
	
	\noindent\emph{Proof of the claim.} Assume to the contrary that $\widehat{K}$ is a subgraph of  $\widehat{G}^*$ isomorphic to $(K_4,-)$. Since neither $\widehat{G}$ nor $\widehat{K}^s_{10;3}$ contains a subgraph isomorphic to $(K_4,-)$, we have $1\leq |V(\widehat{H}) \cap V(\widehat{K})|\leq 3$. Based on the size of $V(\widehat{H}) \cap V(\widehat{K})$, we consider three cases.
	\begin{itemize}
		\item[(1)] $|V(\widehat{H}) \cap V(\widehat{K})|=1$.
		
		In this case, $\widehat{K}\cap (\widehat{G}-\widehat{H})$  is a (negative) triangle, and hence there are at least three edges connecting $\widehat{H}$ and the triangle. Thus 
		$$p(\widehat{G}[V(H)\cup V(K_4)])\leq p(\widehat{H})+3\times 3-2\times 6=p(\widehat{H})-3\leq -2.$$

		\item[(2)] $|V(\widehat{H}) \cap V(\widehat{K})|=2$.
		
		The intersection $\widehat{K}\cap (\widehat{G}-\widehat{H})$ induces a $K_{2}$ which is connected to $\widehat{H}$ by at least four edges. Hence $$p(\widehat{G}[V(H)\cup V(K_4)])\leq p(\widehat{H})+3\times 2-2\times 5=p(\widehat{H})-4\leq -3.$$

		\item[(3)] $|V(\widehat{H}) \cap V(\widehat{K})|=3$.
		
		Similarly, noting that the  intersection $\widehat{K}\cap (\widehat{G}-\widehat{H})$ is a vertex connected to $\widehat{H}$ with at least three edges, we have
		$$p(\widehat{G}[V(H)\cup V(K_4)])\leq p(\widehat{H})+3\times 1-2\times 3=p(\widehat{H})-3\leq -2.$$ 
		
	\end{itemize}
	
	In all cases, we find a subgraph including $\widehat{H}$ with potential value at most $-2$. This either contradicts \Cref{lem:key}(I), or  \Cref{lem:key}(II), or the choice of $\widehat{H}$, hence completing the proof of Claim~\ref{clm:K_4}.
	\bigskip
	
	It follows from \Cref{clm:digon} and \Cref{clm:K_4} that the $10/3$-critical graph $\widehat{G}'$ satisfies all the conditions of Theorem~\ref{thm:main}. Since $|V(\widehat{G}')|<|V(\widehat{G})|$, we have $p(\widehat{G}')\leq -1$.
	
Since potential is a linear function, we have 	$$p(\widehat{G}[V(\widehat{G}'-\phi(\widehat{H}))+V(\widehat{H})])\leq p(\widehat{G}')-p(\widehat{G}'\cap \phi(\widehat{H}))+p(\widehat{H})\leq p(\widehat{G}')-1+p(\widehat{H})\leq p(\widehat{H})-2\leq -1, $$
	which again either contradicts  \Cref{lem:key}(I), or  \Cref{lem:key}(II), or the choice of $\widehat{H}$. This completes the proof.
\end{proof}

\begin{lemma}\label{lem:adding}
	Assume $\widehat{H}\subseteq\widehat{G}$ such that $|V(\widehat{H})|\leq |V(\widehat{G})|-2$. Let $\widehat{H}+e$ be a signed graph obtained from $\widehat{H}$ by adding an edge $e$. Then 
	\begin{itemize}
		\item either $\widehat{H}+e$ is $10/3$-colorable,
		\item or $\widehat{H}+e$ contains one of $(K_4,-)$, $\widetilde{K}_2$, $\widehat{T}^+$ as a subgraph.
	\end{itemize}
\end{lemma}

\begin{proof}
	Suppose, to the contrary, that $\widehat{H}+e$ is not $10/3$-colorable and that none of its $10/3$-critical subgraphs is isomorphic to $(K_4,-)$, $\widetilde{K}_2$, or $\widehat{T}^+$. Let $\widehat{H}'$ be a $10/3$-critical subgraph of $\widehat{H}+e$.

	
	If $\widehat{H}\cong \widehat{T}$, then we are done by \Cref{lem:Few_Vertex_Critical}. Since $|V(\widehat{H}')|\leq |V(\widehat{G})|-2$, by \Cref{lem:key}(III), we have $p(\widehat{H}'-e)\geq 2$. Therefore,  $p(\widehat{H}')\geq 0$, which contradicts the minimality of  $\widehat{G}$.
\end{proof}


\begin{lemma}\label{lem:NoPositiveK3}
	There is no positive triangle in $\widehat{G}$.
\end{lemma}
\begin{proof}
	Towards a contradiction, assume $uvw$ is a positive triangle. Without loss of generality, assume all the edges of $uvw$ are positive. For a positive edge $xy$, let $\widehat{G}_{xy}$ be the signed graph obtained from $\widehat{G}$ by contracting the edge $xy$.
	
	We claim that one of $\widehat{G}_{uv}$, $\widehat{G}_{uw}$ and $\widehat{G}_{vw}$ has no digon. Suppose not, then there is a vertex $v'$ connected to $u$ and $w$ with different signs, $u'$ connected to $w$ and $v$ with different signs, and similarly, $w'$ connected to $u$ and $v$ with different signs. First we claim that $u'$, $v'$, $w'$ are distinct vertices. If not, by symmetries, assume $u'=v'$. Then the subgraph induced by $\{u',u,v,w\}$ is a $K_4$, but noting that $p(K_{4})=0$, this contradicts \Cref{lem:key}. Thus the subgraph induced by the six vertices is isomorphism to $\widehat{H}_1$ in \Cref{fig:10/3-colorable graphs}. However, $p(\widehat{H}_1)=0$, and by \Cref{lem:key}, we must have $\widehat{G}\cong \widehat{H}_1$. This is a contradiction because $\widehat{H}_1$ is 10/3-colorable by \Cref{lem:10/3-colorable_graphs}.
	
	By symmetries, we assume $\widehat{G}_{uv}$ contains no digon. We claim, furthermore, that $\widehat{G}_{uv}$ contains no subgraph  isomorphic to $(K_4,-)$. Let $\widehat{K}$ be a $(K_4,-)$ subgraph of $\widehat{G}_{uv}$. Since $\widehat{G}$ doesn't contain  $(K_4,-)$, we have $u*v\in V(\widehat{K})$. Depending on whether $w$ is a vertex of $\widehat{K}$ or not, we have two cases:
	
	$$p(\widehat{G}[(V(\widehat{K})-u*v)\cup \{u, v\}])\leq 3\times 5-2\times 8=-1, \text{ if } w\in V(\widehat{K});$$
	$$p(\widehat{G}[(V(\widehat{K})-u*v)\cup \{u, v, w\}])\leq 3\times 6-2\times 9=0, \text{ if } w\notin V(\widehat{K}).$$

	The former is impossible by lemma \ref{lem:key}. For the latter to be possible, first of all, equality must hold, hence there are only three edges connecting $(V(\widehat{K})-u*v)$ and $\{u, v, w\}$, secondly, we must have $\widehat{G}[(V(\widehat{K})-u*v)\cup \{u, v, w\}]=\widehat{G}$. However, the subgraph induced by $(V(\widehat{K})-u*v)\cup \{u, v, w\}$ is  $\widehat{H}_3$ in \Cref{fig:10/3-colorable graphs}, which is $10/3$-colorable.

	Therefore, there is a $10/3$-critical graph $\widehat{G}'$  in $\widehat{G}_{uv}$ that satisfies all the conditions of Theorem \ref{thm:main}. By the induction hypothesis, $p(\widehat{G}')\leq -1$. Observe that $u*v\in V(\widehat{G}')$. Again we consider two cases based on whether $w$ is a vertex of $\widehat{G}'$ or not. 

\begin{align*}
\text{If } w \in V(\widehat{G}'), \text{ then } &
\begin{aligned}[t]
p\big(\widehat{G}[V(G')-u*v+\{u, v\}]\big) 
&\leq p(\widehat{G}') - p(u*v) + p(uv) - 2 \\ 
&\leq p(G') - 1 \leq -2.
\end{aligned} \\[1ex]
\text{If } w \notin V(\widehat{G}'), \text{ then } &
\begin{aligned}[t]
p\big(\widehat{G}[V(G')-u*v+\{u, v, w\}]\big) 
&\leq p(\widehat{G}') - p(u*v) + p(uvw) \\ 
&\leq p(G') - 3 + 3 \times 3 - 2 \times 3 \leq -1.
\end{aligned}
\end{align*}

	In either of the cases, we get a subgraph of $\widehat{G}$ with potential at most $-1$, contradicting \Cref{lem:key}.
\end{proof}

\begin{lemma}\label{lem:PairTriangles}
	If an edge $uv$ is in two triangles, say $uvx$ and $uvy$, then there exists a vertex $w$ adjacent to both $x$ and $y$, such that the subgraph induced by $\{x,y,u,v,w\}$ is isomorphic to $\widehat{T}$.
\end{lemma}

\begin{proof}
	Suppose there is no such a vertex $w$. It follows from \Cref{lem:NoPositiveK3} that $uvx$ and $uvy$ are negative triangles. Hence we may assume all edges of $uvx$ and $uvy$ are negative. By the potential value of  the pair of  adjacent triangles, $xy\notin E(\widehat{G})$.   Let $\widehat{G}_{xy}$ be the resulting graph after identifying $x$ and $y$.
	Since $w$ does not exist, $\widehat{G}_{xy}$ has no digon. 
	
	First we show that $\widehat{G}_{xy}$ does not contain any subgraph which is isomorphic to $(K_4,-)$. In fact, we show there is no $K_4$ in the underlying graph of $\widehat{G}_{xy}$. 
	Observe that, by \Cref{lem:key}, $G$ contains no $K_4$. Thus $x*y\in V(\widehat{K})$, where $\widehat{K}$ is a signed $K_4$ in $\widehat{G}_{xy}$. Based on the size of $\{u,v ,x*y\}\cap V(\widehat{K})$, we consider three cases.
	
	\begin{itemize}
		\item[(1)] $|\{u, v, x*y\}\cap V(\widehat{K})|=1$.
		There is a  triangle in $\widehat{K}\cap (\widehat{G}-\{u, v, x, y\})$, and hence, in $\widehat{G}$, there are at least three edges connecting $\{x,y\}$ and the triangle. Thus 
		$$p(\widehat{G}[(V(\widehat{K})-x*y)\cup \{u, v, x, y\}])\leq 3\times 7-2\times 11=-1,$$ which  contradicts   \Cref{lem:key}.
		
		\item[(2)]  $|\{u, v, x*y\}\cap V(\widehat{K})|=2$. 
		In this case, $\widehat{K}\cap (\widehat{G}-\{u, v, x, y\})$ is an edge, and hence there are at least four edges connecting it to $\{u,v, x,y\}$. Thus
		$$p(\widehat{G}[(V(\widehat{K})-x*y)\cup \{u, v, x, y\}])\leq 3\times 6-2\times 10=-2,$$ which again contradicts  \Cref{lem:key}.
		\item[(3)] $|\{u, v, x*y\}\cap V(\widehat{K})|=3$.
		The size of the intersection implies that $\widehat{G}[\{u, v, x, y\}\cup (V(\widehat{K})-x*y)]$ contains a signed $K_4$, which also contradicts  \Cref{lem:key}.
	\end{itemize}
	
	Since $\widehat{G}_{xy}$ contains neither a $(K_4,-)$ nor a digon, the signed graph $\widehat{G}_{xy}$ contains a $10/3$-critical signed graph $\widehat{G}'$ which satisfies all the conditions of \Cref{thm:main}. By induction hypothesis, $p(\widehat{G}')\leq -1$. Since $\widehat{G}$ is $10/3$-critical,  $x*y\in V(\widehat{G}')$. 
	
	Let $U=V(\widehat{G}')\cap \{u,v,x*y\}$. The graph induced by $U$ is either $K_1$ or $K_2$ or $K_3$. Hence $p(\widehat{G}'[U])$  is either $3$ or $4$. We observe: $$p(\widehat{G}[V(\widehat{G}')-U+\{u,v,x,y\}])\leq p(\widehat{G}')-p(\widehat{G}'[U])+3\times 4-2\times 5 \leq p(\widehat{G}')-3+2\leq -2, $$
	which contradicts \Cref{lem:key}. 
\end{proof}

To proceed with the discharging method, it is necessary to examine the local structure of low-degree vertices. A \emph{$k$-vertex} is a vertex of degree $k$. A \emph{$k^+$-vertex} (respectively, \emph{$k^-$-vertex}) is a vertex of degree at least (respectively, at most) $k$. A \emph{$k$-neighbor} of a vertex $v$ is a neighbor of $v$ which is of degree $k$. Similarly, we use $k^+$-neighbor and $k^-$-neighbor to refer to neighbors of degree at least or at most $k$, respectively. A \emph{$(k,d)$-edge} is an edge connecting a $k$-vertex and a $d$-vertex.

Our first set of forbidden configurations is as follows.

\begin{lemma}\label{lem:TwoNeighborCount}
	Every $9^-$-vertex has at least one $3^+$-neighbor. Moreover, every $5^-$-vertex has at least two $3^+$-neighbors. 
\end{lemma} 
\begin{proof}
	
	Assume $v$ is a $9^-$-vertex and all the neighbors of $v$ are the degree two. Let $\phi$ be a $10/3$-coloring of  $\widehat{G}'=\widehat{G}-N_{\widehat{G}}[v]$. By criticality, such a coloring exists. Let $X=N_{\widehat{G}}(v)$. Assume $x\in X$ and $N_{G}(x)=\{v,u\}$. By \Cref{lem:2-vertex}, the color $\phi(u)$  forbids only one color ($\phi(u)$ or $-\phi(u)$) on the list of $v$. With respect to the coloring of $N_{\widehat{G}'}(X)$, the number of permissible colors for $v$ is at least $10-d_{G}(v)>0$. Therefore, the coloring $\phi$ can be extended to $\widehat{G}$.
	
	Assume $v$ is a $5^-$-vertex and $v$ has at most one $3^+$-neighbor. Let $X=\{x\mid xv\in E(\widehat{G}), d(x)=2\}$. Then $|X|\geq d(v)-1$. Let $\phi$ be a $10/3$-coloring of  $\widehat{G}'$, where $\widehat{G}'=\widehat{G}-X-v$. 
	By \Cref{obs:counting_partial_coloring-K2} and \Cref{lem:2-vertex}, with respect to the coloring of $3^+$-neighbor of $v$ and $N_{\widehat{G}'}(X)$, the number of permissible colors for $v$ is at least $10-(d_{G}(v)-1)-5=6-d_G(v)>0$. In other words, $v$ has at least one available color, and hence, the coloring $\phi$ can be extended to $G$.
\end{proof}

A $3$-vertex with a 2-neighbor is called  \emph{poor}. We say a vertex $x$ of $\widehat{G}$ is \emph{rich} if either $d_{G}(x)=4$ and $x$ has at most one 2-neighbor or $d_{G}(x)\geq 5$. We call the pairs $\{v_{1}, v_{3}\}$ and $\{v_{1}, v_{4}\}$ in a copy of $\widehat{T}$ \emph{pseudo-neighbors} (noting that they cannot be adjacent because of criticality).

\begin{lemma}\label{lem:3_vertex_neighbors}
	If $uv$ is a $(3,2)$-edge of $\widehat{G}$ with $N(u)=\{v,x,y\}$, $N(v)=\{u,z\}$, then $z\notin \{x,y\}$. Moreover, either $xy\in E(\widehat{G})$ or there are vertices $w_1,w_2$ such that  $\{u, x, y, w_1, w_2\}$ induces a copy of $\widehat{T}$ and $z\notin \{w_1, w_2\}$.
\end{lemma}

\begin{proof}
	Without loss of generality, noting that every triangle must be negative (\Cref{lem:NoPositiveK3}), we assume that all edges incident to  $u$ or $v$ are negative. 
	
	Towards a contradiction, suppose $z=y$.  By \Cref{lem:2-vertex}, it suffices to prove that there exists a mapping $\phi$ from $\widehat{G}-v$ to $K_{10;3}^s$ with $\phi(u)\neq -\phi(z)$. Since $\widehat{G}-v$ is still $10/3$-colorable, there exists a $10/3$-coloring $\psi$ of $\widehat{G}-\{u, v\}$ such that $N^-_{K_{10;3}^s}(\psi(z))\cap N^-_{K_{10;3}^s}(\psi(x))\neq \emptyset$. Moreover, if the number of  colors in this intersection is exactly one, then that color is at distance three from both $\psi(z)$ and $\psi(x)$. There is a choice $c$ of $u$ such that $c\in N^-_{K_{10;3}^s}(\psi(z))\cap N^-_{K_{10;3}^s}(\psi(x))$ and $c \neq -\psi(z)$. The coloring $\psi$ together with $c$ is the required coloring $\phi$ of $\widehat{G}-v$. Therefore, the neighbors of $u$ and $v$ are distinct vertices.

	Suppose $xy\notin E(\widehat{G})$. Let $\widehat{G}'=\widehat{G}-\{u,v\}+xy$ assigning the positive sign to $xy$. If $\widehat{G}'$ contains a $(K_4,-)$, then the pair $x, y$ should be two vertices of $(K_4,-)$. Let $w_1,w_2$ be the other two vertices of this $(K_4,-)$. Then $\{u, x, y, w_1, w_2\}$ induces a copy of $\widehat{T}$.
	
	If $\widehat{G}'$ contains a $\widehat{T}^+$, then, since $\widehat{T}^+$  is not a subgraph of $G$, $xy$ must be in the $\widehat{T}^+$. The potential value of the subgraph induced by $V(\widehat{T}^+)\cup u$ is $3\times 6 -2\times 9=0$, which, noting that $v$ is not in the set $V(\widehat{T}^+)\cup u$, contradicts \Cref{lem:key}.
	
	By \Cref{lem:adding}, there exists a mapping $\phi$ from $\widehat{G}'$ to $K_{10;3}^s$. By \Cref{lem:counting_partial_coloringK3}, $|N^-_{K_{10;3}^s}(\phi(x))\cap N^-_{K_{10;3}^s}(\phi(y))|\ge 3$. Combined with \Cref{lem:2-vertex}, $z$ forbids only one color (namely $-\phi(z)$) on $u$. Thus we have at least two choices for extending $\phi$ on $u$ and $v$. This contradicts the choice of $\widehat{G}$.
	
	Assume $z=w_1$, then $p(\widehat{T}+v)=3\times 6-2\times9=0$. By \Cref{lem:key}, $\widehat{G}=\widehat{G}[u,v,x,y,w_1,w_2]$. However, $\widehat{G}\cong \widehat{H}_4$ or $\widehat{G}\cong \widehat{H}'_4$, where $\widehat{H}_4$ and $\widehat{H}'_4$ are depicted in \Cref{fig:10/3-colorable graphs}, and by \Cref{lem:10/3-colorable_graphs}, $\widehat{G}$ is 10/3-colorable. This completes the proof.
\end{proof}

\begin{lemma}\label{lem:32-edge}
	If $uv$ is a $(3,2)$-edge of $\widehat{G}$ with $N(u)=\{v,x,y\}$, $N(v)=\{u,z\}$,  then at least two of the neighbors and pseudo-neighbors of $u$  are rich. 
\end{lemma}

\begin{proof}
	Without loss of generality,  we assume all edges incident to  $u$ or $v$ are negative. By \Cref{lem:3_vertex_neighbors}, either $xy\in E(\widehat{G})$ or there are vertices $w_1,w_2$ such that $\{u, x, y, w_1, w_2\}$ induces a copy of $\widehat{T}$. We discuss two cases separately. 
	\setcounter{claim}{0}
	
	\noindent
	\textbf{Case 1} $xy\in E(\widehat{G})$.
	
	First we show that $d(x)\ge 4$ and $d(y)\ge 4$. By \Cref{lem:TwoNeighborCount}, we have $d(x), d(y)\geq 3$. Towards a contradiction, assume $d_G(x)=3$. Let $w$ be the other neighbor of $x$. Consider a $10/3$-coloring $\phi$ of $\widehat{G}-\{u,v,x\}$ which can be extended to $\widehat{G}-\{u,v\}$ ($\phi$ exists because $\widehat{G}$ is critical). Assume $\phi(y)=1$, depending on the sign of $xw$, we have two cases: If $xw$ is negative (positive, respectively), then $\phi(w)\neq -1$ ($\phi(w)\neq 1$). By symmetries, we may assume $\phi(w)\in \{1,2,3,4,5\}$ ($\phi(w)\in \{2,3,4,5,-1\}$). We choose $\phi(x)=-3$ ($\phi(x)=4$). With respect to $\phi(x)$ and $\phi(y)$, $u$ can be colored by $4$ or $5$ ($-2$ or $-3$), only one of which might be forbidden by $\phi(z)$ because of \Cref{lem:2-vertex}. Hence, we have a 10/3-coloring of $\widehat{G}$, which is a contradiction.    
	
	To complete the proof in this case, we show that both $x$ and $y$ are rich.
	
	By symmetry, assume $d(x)=4$ and that $x$ has two 2-neighbors, say $w_1$ and $w_2$. Let $\phi$  be a coloring of $\widehat{G}-\{u,v,x, w_1, w_2\}$. 
	
	It follows from \Cref{lem:TwoNeighborCount} that $z\neq w_1, w_2$. For extending $\phi$ to $x$, each of the neighbors of $w_1$ and $w_2$ forbids one color by \Cref{lem:2-vertex}, and $\phi(y)$ forbids five colors by \Cref{obs:counting_partial_coloring-K2}. In total we have $|\phi(x)|\ge 3$.  Similarly, as each of $u$ and $v$ have only one colored neighbor, we have $|\phi(u)|=|\phi(v)|=5$. Applying \Cref{lem:K2P3Coloring}, there are choices for the set $\{u,v,x, w_1, w_2\}$.
	
	\bigskip
	\noindent
	\textbf{Case 2} $xy\notin E(\widehat{G})$, and there are vertices $w_1,w_2$ such that $\{u, x, y, w_1, w_2\}$ induces a copy of $\widehat{T}$.
	
	Assume $w_1y$ and $w_2y$ are positive edges and the other edges of $\widehat{T}$ are negative. We show that at least two vertices in $\{x, y,w_1,w_2\}$ are rich.
	
	By \Cref{lem:3_vertex_neighbors}, $z$ is not $w_1$ or $w_2$. As the degree of $x, y, w_1, w_2$ is at least three, it is enough to prove that at least two vertices in $\{x, y,w_1,w_2\}$ are of degree  at least four. 
	
	Since $\widehat{G}$ is 2-connected, we may assume that only one of $x, y, w_1, w_2$ has neighbors out of the $\widehat{T}$. By symmetries, we assume it is either $y$ or $w_1$. Let $\widehat{G}'=\widehat{G}-\{u,v,x, y, w_1,w_2\}$. There exists a mapping $\phi$ from $\widehat{G}'$ to $K_{10;3}^s$ by criticality of $\widehat{G}$. We may choose $\phi$ such that if $d(y)\ge 4$, then $\phi(y)=1$ is possible, and if $d(w_1)\ge 4$, then $\phi(w_1)=-5$ is possible. In both cases, we can extend the coloring to $\phi(y)=1$, $\phi(w_1)=-5$, $\phi(w_2)=3$, and $\phi(x)=-3$. This leaves two options for $u$, $4$ and $5$, only one of which can be forbidden by $\phi(z)$.  
\end{proof}

\begin{lemma}\label{lem:Cubic}
	$\widehat{G}$ is not a cubic signed graph.
\end{lemma}
\begin{proof}
	Assume to the contrary that $\widehat{G}$ is a cubic signed graph. 
	
	\setcounter{claim}{0}
	
	\begin{claim}
		There is no pair of adjacent triangles in $\widehat{G}$.
	\end{claim}
	\noindent\emph{Proof of the claim.} Assume $\{uvx, uvy\}$ is a pair of adjacent triangles. It follows from \Cref{lem:PairTriangles}  that there exists a vertex $w$ such that $\{u,v,x, y, w\}$ induces an isomorphic copy of $\widehat{T}$. Since $\widehat{G}$ is cubic, $w$ is a cut vertex, a contradiction.

	\begin{claim}
		There is no triangle in $\widehat{G}$. 
	\end{claim}
	
	\noindent\emph{Proof of the claim.}
	Assume $\widehat{G}$ contains a triangle $uvw$, which must be negative by \Cref{lem:NoPositiveK3}. Assume $u',v'$ and $w'$ are the corresponding neighbors of $u$, $v$ and $w$.  By Claim 1, $u', v', w'$ are distinct. Without loss of generality, assume that all edges incident to  $u,v,w$ are negative. It follows from \Cref{lem:10/3-colorable_graphs}  that $u',v',w'$ induce no triangle; if they do, then it must be a negative triangle, and hence $\widehat{G}$ is either $\widehat{H}_2$ or $\widehat{H}'_2$ of \Cref{fig:10/3-colorable graphs}. By symmetries, we may assume $u'v'\notin E(\widehat{G})$.

	Let $\widehat{G}'$ be the signed graph obtained from $\widehat{G}$ by removing the three vertices $u,v,w$ and adding the edge $u'v'$ with a negative sign. Since $\widehat{G}$ contains no pair of adjacent triangles, $\widehat{G}'$ contains neither $(k_4,-)$ nor $\widehat{T}^+$. By the choice of $u'v'$, there is no $\widetilde{K}_2$ in $\widehat{G}'$ either. Hence, by \Cref{lem:adding}, $\widehat{G}'$ admits a $10/3$-coloring $\phi$. 	
	By symmetries, we may assume $\phi(u')=1$ and $\phi(v')\in \{4,5,-1\}$. We extend the coloring by setting $\phi(v)=1$ and $\phi(u)\in \{-3,5\}$. The choice of $-3$ leaves us with the options $\{4,5\}$ for $w$, the choice of $5$ leaves us with the options $\{-3,-4\}$ for $w$. As $\phi(w')$ can only forbid a consecutive set of five colors, we can find a matching choice for $w$. Thus the coloring $\phi$ can be extended to $\widehat{G}$, which is a contradiction, hence verifying Claim 2.
	
	\bigskip
	
	We now consider an edge $uv\in E(\widehat{G})$. Let  $w_1, w_2$ be the other neighbors of $u$ and $w_3, w_4$ be the other neighbors of $v$. As $\widehat{G}$ is triangle-free, they are distinct and $w_1w_2, w_3w_4\notin E(\widehat{G})$. Thus we may assume edges incident to $u,v$ are all negative. Let $W=\{w_1,w_2, w_3, w_4\}$. Depending on the subgraph induced by $W$, we consider four cases, which are shown in \Cref{fig:CubicGraph}. Note that a black thick edge indicates an edge with unknown sign, and a black dotted edge indicates a potential edge whose existence is uncertain in \Cref{fig:CubicGraph}.
	\begin{figure}
		\centering
		\begin{minipage}[h]{0.4\textwidth}
			\centering
			\begin{tikzpicture}[scale=0.8,
				baseline=(current bounding box.center),
				node style/.style={circle,draw,inner sep=0.4pt, minimum size=5mm},
				dashed edge/.style={blue, dashed},
				edge/.style={red, thick}
				]
				\node[node style] (w1) at (0, 0) {$w_1$};
				\node[node style] (w2) at (2,0) {$w_2$};
				\node[node style] (w3) at (3,0) {$w_3$};
				\node[node style] (w4) at (5,0) {$w_4$};
				\node[node style] (u) at (1,2.5) {$u$};
				\node[node style] (v) at (4,2.5) {$v$};
				
				\draw[edge] (w3) -- (v)--(w4);
				\draw[edge] (w1) -- (u)--(w2);
				\draw[edge] (v) -- (u);
				\draw[very thick] (w2)--(w3);
			\end{tikzpicture}
			\caption*{Case 1}
		\end{minipage}
		\hfill
		\begin{minipage}[h]{0.4\textwidth}
			\centering
			\begin{tikzpicture}[scale=0.8,
				baseline=(current bounding box.center),
				node style/.style={circle,draw,inner sep=0.5pt, minimum size=5mm},
				dashed edge/.style={blue, dashed},
				edge/.style={red, thick}
				]
				\node[node style] (w1) at (0, 0) {$w_1$};
				\node[node style] (w2) at (2,0) {$w_2$};
				\node[node style] (w3) at (3,0) {$w_3$};
				\node[node style] (w4) at (5,0) {$w_4$};
				\node[node style] (u) at (1,2.5) {$u$};
				\node[node style] (v) at (4,2.5) {$v$};
				
				\draw[edge] (w3) -- (v)--(w4);
				\draw[edge] (w1) -- (u)--(w2);
				\draw[edge] (v) -- (u);
			\end{tikzpicture}
			\caption*{Case 2}
		\end{minipage}
		\\
		\vspace{5mm}
		\begin{minipage}[h]{0.4\textwidth}
			\centering
			\begin{tikzpicture}[scale=0.8,
				baseline=(current bounding box.center),
				node style/.style={circle,draw,inner sep=0.5pt, minimum size=5mm},
				dashed edge/.style={blue, dashed},
				edge/.style={red, thick}
				]
				\node[node style] (w1) at (0, 0) {$w_1$};
				\node[node style] (w2) at (2,0) {$w_2$};
				\node[node style] (w3) at (3,0) {$w_3$};
				\node[node style] (w4) at (5,0) {$w_4$};
				\node[node style] (u) at (1,2.5) {$u$};
				\node[node style] (v) at (4,2.5) {$v$};
				
				\draw[edge] (w3) -- (v)--(w4);
				\draw[edge] (w1) -- (u)--(w2);
				\draw[edge] (v) -- (u);
				\draw[very thick] (w1) to[out=-60, in=-120] (w3);
				\draw[very thick] (w1) to[out=-60, in=-120] (w4);
				\draw[dotted, very thick] (w2)--(w3);
				\draw[dotted, very thick] (w2) to[out=-60, in=-120] (w4);
				
			\end{tikzpicture}
			\caption*{Case 3}
		\end{minipage}
		\hfill
		\begin{minipage}[h]{0.4\textwidth}
			\centering
			\begin{tikzpicture}[scale=0.8,
				baseline=(current bounding box.center),
				node style/.style={circle,draw,inner sep=0.5pt, minimum size=5mm},
				dashed edge/.style={blue, dashed},
				edge/.style={red, thick}
				]
				\node[node style] (w1) at (0, 0) {$w_1$};
				\node[node style] (w2) at (2,0) {$w_2$};
				\node[node style] (w3) at (3,0) {$w_3$};
				\node[node style] (w4) at (5,0) {$w_4$};
				\node[node style] (u) at (1,2.5) {$u$};
				\node[node style] (v) at (4,2.5) {$v$};
				
				\draw[edge] (w3) -- (v)--(w4);
				\draw[edge] (w1) -- (u)--(w2);
				\draw[edge] (v) -- (u);
				\draw[very thick] (w2)--(w3);
				\draw[very thick] (w1) to[out=-60, in=-120] (w4);
			\end{tikzpicture}
			\caption*{Case 4}
		\end{minipage}
		\caption{Four Cases of the subgraph induced by $W$}
		\label{fig:CubicGraph}
	\end{figure}
	
	Assume that $x_i, y_i$ are the other neighbors of $w_i$ for $i \in [4]$. 
Note that $x_i$ may possibly be equal to $x_j$ or $y_j$ for $i \neq j$. 
 First we show that $\{x_{1}, y_{1}\} \neq \{x_{2}, y_{2}\}$.
 Otherwise, the underlying graph of $\widehat{G}[x_{1}, y_{1}, w_{1}, w_{2}, u]$ is $K_{2,3}$. Since $\widehat{G}$ is cubic, the lists obtained on this $K_{2,3}$ after a $10/3$-coloring of $\widehat{G}-\{x_{1}, y_{1}, w_{1}, w_{2}, u\}$ satisfy the condition of \Cref{lem:K_{2,3}}, and hence can be extended to $\widehat{G}$. 
Similarly we have $\{x_{3}, y_{3}\} \neq \{x_{4}, y_{4}\}$.

	\textbf{Case 1.} $E(\widehat{G}[W])=w_2w_3$. 
	
	First we claim that $uvw_3w_2$ is a negative 4-cycle. Suppose not,  then the edge  $w_2w_3$ is negative. Let $\widehat{G}'$ be obtained from $\widehat{G}$ by deleting vertices $u$, $v$ and adding the edges $w_1w_3$ and $w_2w_4$, each with a negative sign. Since $\widehat{G}'$ and all its subgraphs are of potential at least 0, $\widehat{G}'$ admits a 10/3-coloring $\phi$. The coloring  $\phi$ can be extended to $\widehat{G}$ by assigning $\phi(u)=\phi(w_3)$ and $\phi(v)=\phi(w_2)$, which is a contradiction. Hence $w_2w_3$ is a positive edge.
	
	Suppose that $\{w_{1}, x_{1}, w_{4}, y_{1}\}$ doesn't induce a negative 4-cycle. Let $\widehat{G}'$ be obtained from $\widehat{G}$ by deleting the vertices in $\{u,v\}\cup W$ and adding the edges $x_1y_1$ and $x_4y_4$ such that each of $w_1x_1y_1$ and $w_4x_4y_4$ is a positive triangle.  This is possible because we assume $\{w_{1}, x_{1}, w_{4}, y_{1}\}$ doesn't induce a negative 4-cycle.  As $\widehat{G}'$ is still cubic and all its subgraphs are of potential at least 0, $\widehat{G}'$ admits a 10/3-coloring $\phi$.  We claim $\phi$ can be extended to a $10/3$-coloring of $\widehat{G}$, which is a contradiction. That is because to  extend $\phi(x_1)$ and $\phi(y_1)$ to $w_1$, by \Cref{lem:counting_partial_coloringK3}, $w_1$ has three choices. Based on the three choices of $w_1$, by \Cref{lem:a+4}, we have 7 choices for $u$. Similarly, we have 7 choices for $v$ based on $w_4$. As in this case,  $w_{2}$ and $w_{3}$, each has only one colored neighbor, they have a list of $5$ consecutive colors. By  \Cref{lem:4-cycle}, the coloring $\phi$ can be extended to a $10/3$-coloring of $\widehat{G}$, a contradiction. 

	Until now we have proved that  if for an edge $uv$, its four neighbors $w_1,w_2,w_3,w_4$  induce an edge, say $w_{2}w_{3}$, then the vertices $w_{1}$, $w_{4}$ have two common neighbors, say $x_{1}, y_{1}$, such that $\{w_{1}, w_{4}, x_{1}, y_{1}\}$ induces a negative 4-cycle.  If there is a matching from $w_{2}, w_{3}$ to $x_{1}, y_{1}$, by symmetries, assume $w_{2}x_{1}, w_{3}y_{1}\in E(\widehat{G})$, then the subgraph induced by $\{u, v, x_{1}, y_{1}\}\cup W$ is already cubic. Depending on the signs of these edges, we have four possibilities for $\widehat{G}$, which are presented in \Cref{fig:8vertices}, each together with a $10/3$-coloring. Hence, one of $w_{2}, w_{3}$ has no neighbor in $x_{1}, y_{1}$. By symmetry, assume $w_{2}$ is such a vertex.
	
	We now observe that the four neighbors $w_{1}, v, x_{2}, w_{2}$ of $uw_{2}$ also induce a unique edge, which is $vw_{3}$.  We conclude that the other two   $w_{1}, x_{2}$ have  two common neighbors: $x_{1}, y_{1}$. Over all in the subgraph induced by $\{u, v, x_{1}, x_{2}, x_{3}, y_{1}\}\cup W$, all of the vertices except $w_{3}$ are of degree 3. Since $\widehat{G}$ is cubic, this implies that $w_{3}$ is a cut vertex which is a contradiction to minimality of $\widehat{G}$.

	\begin{figure}
	\centering
	\begin{minipage}[h]{0.4\textwidth}
	\begin{tikzpicture}[scale=0.5,
		baseline=(current bounding box.center),
		node style/.style={circle,draw,inner sep=0.5pt, minimum size=5mm},
		dashed edge/.style={blue, dashed, thick},
		edge/.style={red, thick}
		]
		
		\node[node style] (w1) at (0, 0) {$w_1$};
		\node[left=2pt of w1] {-2};
		\node[node style] (w4) at (6, 0) {$w_4$};
		\node[right=2pt of w4] {-3};
		\node[node style] (u) at (0, 6) {$u$};
		\node[left=2pt of u] {1};
		\node[node style] (v) at (6, 6) {$v$};
		\node[right=2pt of v] {4};
		\node[node style] (x1) at (2, 2) {$x_{1}$};
		\node[left=2pt of x1] {1};
		\node[node style] (y1) at (4, 2) {$y_{1}$};
		\node[right=2pt of y1] {-5};
		\node[node style] (w2) at (2, 4) {$w_2$};
		\node[left=2pt of w2] {-3};
		\node[node style] (w3) at (4, 4) {$w_3$};
		\node[right=2pt of w3] {-2};
		
		\draw[edge] (w1) -- (u) -- (v) -- (w4) -- (x1)--(w1)--(y1);
		\draw[edge] (w2) -- (u);
		\draw[edge] (w3) -- (v);
		\draw[dashed edge] (w2) -- (w3);
		\draw[dashed edge] (w4)--(y1);
		\draw[edge] (x1)--(w2);
		\draw[edge] (y1)--(w3);
	\end{tikzpicture}
	\end{minipage}
		\begin{minipage}[h]{0.4\textwidth}
	\begin{tikzpicture}[scale=0.5,
		baseline=(current bounding box.center),
		node style/.style={circle,draw,inner sep=0.5pt, minimum size=5mm},
		dashed edge/.style={blue, dashed, thick},
		edge/.style={red, thick}
		]
		
		\node[node style] (w1) at (0, 0) {$w_1$};
		\node[left=2pt of w1] {4};
		\node[node style] (w4) at (6, 0) {$w_4$};
		\node[right=2pt of w4] {-2};
		\node[node style] (u) at (0, 6) {$u$};
		\node[left=2pt of u] {1};
		\node[node style] (v) at (6, 6) {$v$};
		\node[right=2pt of v] {4};
		\node[node style] (x1) at (2, 2) {$x_{1}$};
		\node[left=2pt of x1] {-5};
		\node[node style] (y1) at (4, 2) {$y_{1}$};
		\node[right=2pt of y1] {-4};
		\node[node style] (w2) at (2, 4) {$w_2$};
		\node[left=2pt of w2] {-3};
		\node[node style] (w3) at (4, 4) {$w_3$};
		\node[right=2pt of w3] {-2};
		
		\draw[edge] (w1) -- (u) -- (v) -- (w4) -- (x1)--(w1)--(y1);
		\draw[edge] (w2) -- (u);
		\draw[edge] (w3) -- (v);
		\draw[dashed edge] (w2) -- (w3);
		\draw[dashed edge] (w4)--(y1);
		\draw[dashed edge] (x1)--(w2);
		\draw[dashed edge] (y1)--(w3);
	\end{tikzpicture}
	\end{minipage}
	\\
					\vspace{5mm}
				\begin{minipage}[h]{0.4\textwidth}
	\begin{tikzpicture}[scale=0.5,
		baseline=(current bounding box.center),
		node style/.style={circle,draw,inner sep=0.5pt, minimum size=5mm},
		dashed edge/.style={blue, dashed, thick},
		edge/.style={red, thick}
		]
		\node[node style] (w1) at (0, 0) {$w_1$};
		\node[left=2pt of w1] {-2};
		\node[node style] (w4) at (6, 0) {$w_4$};
		\node[right=2pt of w4] {-3};
		\node[node style] (u) at (0, 6) {$u$};
		\node[left=2pt of u] {1};
		\node[node style] (v) at (6, 6) {$v$};
		\node[right=2pt of v] {4};
		\node[node style] (x1) at (2, 2) {$x_{1}$};
		\node[left=2pt of x1] {3};
		\node[node style] (y1) at (4, 2) {$y_{1}$};
		\node[right=2pt of y1] {-5};
		\node[node style] (w2) at (2, 4) {$w_2$};
		\node[left=2pt of w2] {5};
		\node[node style] (w3) at (4, 4) {$w_3$};
		\node[right=2pt of w3] {-2};
		
		\draw[edge] (w1) -- (u) -- (v) -- (w4) -- (x1)--(w1)--(y1);
		\draw[edge] (w2) -- (u);
		\draw[edge] (w3) -- (v);
		\draw[dashed edge] (w2) -- (w3);
		\draw[dashed edge] (w4)--(y1);
		\draw[dashed edge] (x1)--(w2);
		\draw[edge] (y1)--(w3);
	\end{tikzpicture}
	\end{minipage}
	  \hspace{0.02\textwidth} 
				\begin{minipage}[h]{0.4\textwidth}
	\begin{tikzpicture}[scale=0.5,
		baseline=(current bounding box.center),
		node style/.style={circle,draw,inner sep=0.5pt, minimum size=5mm},
		dashed edge/.style={blue, dashed, thick},
		edge/.style={red, thick}
		]
		\node[node style] (w1) at (0, 0) {$w_1$};
		\node[left=2pt of w1] {-2};
		\node[node style] (w4) at (6, 0) {$w_4$};
		\node[right=2pt of w4] {-2};
		\node[node style] (u) at (0, 6) {$u$};
		\node[left=2pt of u] {1};
		\node[node style] (v) at (6, 6) {$v$};
		\node[right=2pt of v] {4};
		\node[node style] (x1) at (2, 2) {$x_{1}$};
		\node[left=2pt of x1] {2};
		\node[node style] (y1) at (4, 2) {$y_{1}$};
		\node[right=2pt of y1] {-4};
		\node[node style] (w2) at (2, 4) {$w_2$};
		\node[left=2pt of w2] {-3};
		\node[node style] (w3) at (4, 4) {$w_3$};
		\node[right=2pt of w3] {-2};
		
		\draw[edge] (w1) -- (u) -- (v) -- (w4) -- (x1)--(w1)--(y1);
		\draw[edge] (w2) -- (u);
		\draw[edge] (w3) -- (v);
		\draw[dashed edge] (w2) -- (w3);
		\draw[dashed edge] (w4)--(y1);
		\draw[edge] (x1)--(w2);
		\draw[dashed edge] (y1)--(w3);
	\end{tikzpicture}
	\end{minipage}
	\caption{The signed graphs}
	\label{fig:8vertices}
\end{figure}
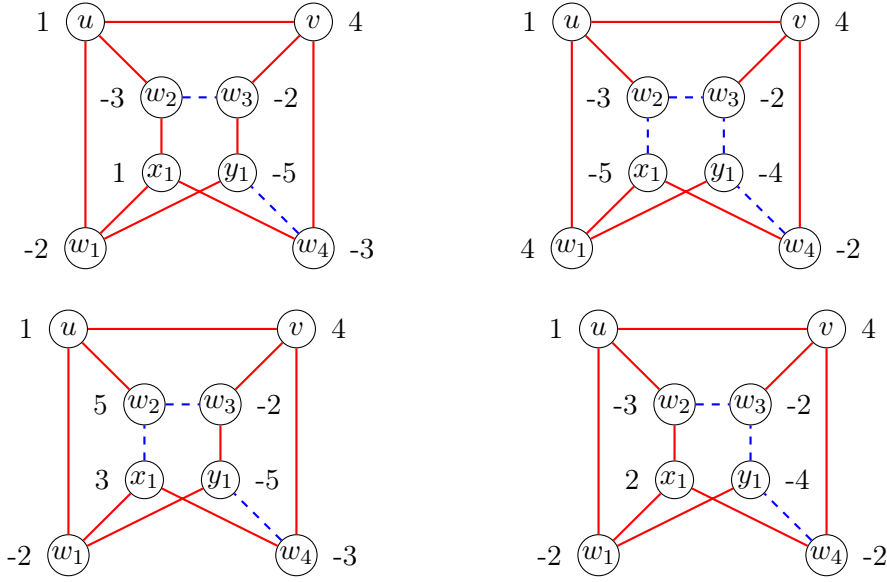

	\textbf{Case 2.}   $W$ is an independent set. 
	
	Suppose $\{w_{i}, x_{i}, y_{i}, w_{j}\}$ dosen't induce a negative 4-cycle for any $i,j\in [4]$. Let $\widehat{G}'$ be obtained from $\widehat{G}$ by deleting vertices in $\{u,v\}\cup W$ and adding the edge set $S=\{x_iy_i\mid i\in [4] \}$ such that each  $w_ix_iy_i$ is a positive triangle. This is possible because $\{w_{i}, x_{i}, y_{i}, w_{j}\}$  induces no negative 4-cycle. As $\widehat{G}'$ is subcubic, it admits a 10/3-coloring $\phi$. We claim that $\phi$ can be extended to a 10/3-coloring of $\widehat{G}$, which is a contradiction. That is because to extend $\phi$ to $w_i$, by \Cref{lem:counting_partial_coloringK3}, each $w_i$ has at least three choices. Based on the three choices of $w_1$, by \Cref{lem:a+4}, we have seven choices for $u$. Similarly, based on the choices of $w_2$, we have seven choices for $u$. Thus altogether we have at least four choices for $u$ based on the colors $\phi(x_1)$,  $\phi(x_2)$, $\phi(y_1)$ and $\phi(y_2)$. Similarly, we have at least four choices for $v$ based on the colors $\phi(x_3)$,  $\phi(x_4)$, $\phi(y_3)$ and $\phi(y_4)$. By \Cref{lem:K2P3Coloring}(a), we have an extension of $\phi$ to $\widehat{G}$.
		
Recall that $\{x_{1}, y_{1}\} \neq \{x_{2}, y_{2}\}$ and $\{x_{3}, y_{3}\} \neq \{x_{4}, y_{4}\}$.  By symmetries, assume that $w_{2}x_{2}y_{2}w_{3}$ is a negative 4-cycle. From Case 1, we conclude that the four neighbors $w_{1},w_{2},w_{3},w_{4}$ of any edge $uv$ cann't induce just one edge in $\widehat{G}$. From previous part of Case 2,  if $\{w_{1},w_{2},w_{3},w_{4}\}$ forms an independent set, then there exist $i \in \{1,2\}$ and $j \in \{3,4\}$ such that $N_{\widehat{G}}(w_{i}) \setminus \{u\} \;=\; N_{\widehat{G}}(w_{j}) \setminus \{v\},$ and the subgraph induced by $\{w_{i},w_{j}\} \cup \bigl(N_{\widehat{G}}(w_{i}) \setminus \{u\}\bigr)$ is a negative $4$-cycle.

Assume $w_{1}x_{2}\in E(\widehat{G})$. Since $\{x_{1}, y_{1}\} \neq \{x_{2}, y_{2}\}$,  $w_{1}y_{2}\notin E(\widehat{G})$. But then Case 1 applies on  the edge $uw_{2}$. Hence $w_{1}x_{2}\notin E(\widehat{G})$. By symmetries, $w_{1}x_{2}, w_{1}y_{2}, w_{4}x_{2}, w_{4}y_{2}\notin E(\widehat{G})$.  
Repeating the argument starting with the edge $uw_{2}$, since their four neighbors, $\{v, w_{1}, x_{2}, y_{2}\}$ is an independent set, there must be a negative 4-cycle  induced by two of these vertices and their neighbors. However, this is not possible because $w_{4}x_{2}, w_{4}y_{2}\notin E(\widehat{G})$ and  $w_{3}x_{2}, w_{3}y_{2}\in E(\widehat{G})$.

	\textbf{Case 3.} $w_1w_3, w_1w_4\in E(\widehat{G})$. 
	
	Observe that $X=\{u,v,w_1,w_3,w_4\}$ induces a signed $K_{2,3}$. Let $\phi$  be a $10/3$-coloring of $\widehat{G}-X$. The induced lists on the vertices in $X$ are a consecutive list of 5 colors for each of $w_{3}$, $w_{4}$, $u$ and the full list of 10 colors for $w_{1}$ and $v$. We may then apply \Cref{lem:K_{2,3}}, which is a contradiction.
	
	Since in Case 3,  we don't restrict other connections between $\{w_{1}, w_{2}\}$ and $\{w_{3}, w_{4}\}$, and by symmetries, we have only one case left. 
	
	\textbf{Case 4.} $E(\widehat{G}[W])=w_1w_4, w_2w_3$. 
	
	First we claim that $w_1w_4$ and $w_2w_3$ are both positive. Assume that $w_1w_4$ is negative.  Let $\widehat{G}'=\widehat{G}-u+w_2w_4$ where $w_2w_4$ is assigned a negative sign. We claim that $\widehat{G}'$ is $10/3$-colorable. If not, then it has a 10/3-critical subgraph $\widehat{H}$. Since $|V(H)|<|V(G)|$, and clearly $\widehat{H}\notin \{\widetilde{K}_2, (K_4,-)\}$, we must have  $p(\widehat{H})\leq -1$. But any subgraph $\widehat{H}'$ of $\widehat{G}'$ satisfies $p(\widehat{H}')\geq 0$ because   $w_{4}$ is the only 4-vertex, however, two of its neighbor ($v$ and $w_{1}$) are  2-vertices. Let $\phi$ be a $10/3$-coloring of $\widehat{G}'$. We can extend $\phi$ to $\widehat{G}$ by assigning $\phi(u)=\phi(w_4)$. 
	
	We conclude that the four neighbors $w_1,w_2,w_3,w_4$ of any edge $uv$ of $\widehat{G}$ induce a matching such that the two 4-cycles are negative. Applying the same on the edge $uw_1$ ($vw_4$, respectively), we must have a vertex $x$ ($y$, respectively) connected to $w_1$ and $w_2$ ($w_3$ and $w_4$, respectively) with different signs. Regarding $w_1w_4$, $xy$ must be an edge such that $xw_1w_4y$ is a negative 4-cycle. Thus $\widehat{G}$ is a sign graph on the cube,  with a signature where every 4-cycle is negative. Noting that any pair of them are switching equivalent, a presentation of it together with a 10/3-coloring is given in \Cref{fig:cube}. 
\end{proof}

\begin{figure}
	\centering
	\begin{tikzpicture}[scale=0.6,
		baseline=(current bounding box.center),
		node style/.style={circle,draw,inner sep=0.5pt, minimum size=5mm},
		dashed edge/.style={blue, dashed, thick},
		edge/.style={red, thick}
		]
		
		\node[node style] (w1) at (0, 0) {$w_1$};
		\node[left=2pt of w1] {-2};
		\node[node style] (w4) at (6, 0) {$w_4$};
		\node[right=2pt of w4] {-3};
		\node[node style] (u) at (0, 6) {$u$};
		\node[left=2pt of u] {1};
		\node[node style] (v) at (6, 6) {$v$};
		\node[right=2pt of v] {4};
		\node[node style] (x) at (2, 2) {$x$};
		\node[left=2pt of x] {-5};
		\node[node style] (y) at (4, 2) {$y$};
		\node[right=2pt of y] {5};
		\node[node style] (w2) at (2, 4) {$w_2$};
		\node[left=2pt of w2] {-3};
		\node[node style] (w3) at (4, 4) {$w_3$};
		\node[right=2pt of w3] {-2};
		
		\draw[edge] (w1) -- (u) -- (v) -- (w4) -- (y)--(x)--(w1);
		\draw[edge] (w2) -- (u);
		\draw[edge] (w3) -- (v);
		\draw[dashed edge] (w1) -- (w4);
		\draw[dashed edge] (x)--(w2)--(w3)--(y);
	\end{tikzpicture}
	\caption{The signed graph on the cube}
	\label{fig:cube}
\end{figure}
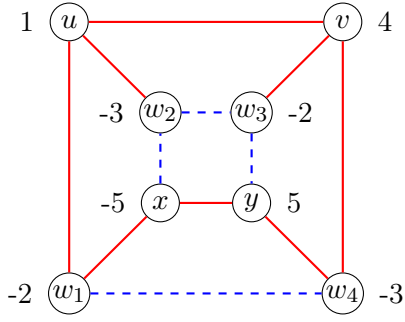

\begin{lemma}\label{lem:Delta4}
	If  $\widehat{G}$ has no $(3,2)$-edge, then $\widehat{G}$ has at least one rich vertex. 
\end{lemma}

\begin{proof}
	Recall that any vertex of degree at least 5 is rich. A vertex of degree 4 is also rich if it has at most one 2-neighbor. As $\widehat{G}$ is not cubic (\Cref{lem:Cubic}), we assume $\Delta(G)=4$. As a 4-vertex with three or four 2-neighbors is forbidden (\Cref{lem:TwoNeighborCount}), we assume that every $4$-vertex of $\widehat{G}$ has exactly two 2-neighbors.

	Consider a 4-vertex $u$ with $x, y, v, w$ being its neighbors, where $v,w$ are degree 2. We claim that $xy\in E(\widehat{G})$. Suppose not. Let $\widehat{G}'=\widehat{G}-\{u,v,w\}+xy$ with assigning a sign to $xy$ such that $uxy$ would be a positive triangle. 
	Since $xy\notin \widehat{G}$, the graph $\widehat{G}'$ does not contain any copy of $\widetilde{K}_{2}$. Suppose that $(K_{4},-)\subseteq \widehat{G}'$. In this case, both $x$ and $y$ are vertices of the $(K_{4},-)$, while $u$ is not contained in it. Moreover, all vertices in the $(K_{4},-)$ must have degree $3$ in $\widehat{G}$; otherwise, there exists a $4$-vertex in $\widehat{G}$ with at least three $3^{+}$-neighbors, contradicting the assumption. Consequently, $K_{4}-xy+u$ forms a block of $\widehat{G}$ and $u$ becomes a cut-vertex of $\widehat{G}$. Suppose $\widehat{T}^{+}\subseteq \widehat{G}'$. Observe that $p(\widehat{T}^{+}-xy)=1$. By \Cref{lem:key}, we know that $\widehat{T}^{+}-xy\cong \widehat{T}$. By the structure of $\widehat{T}$ and $\widehat{T}^{+}$, we have $u\notin V(\widehat{T})$ and then one of $x$ and $y$ is a rich 4-vertex in $\widehat{G}$.
	Therefore, by \Cref{lem:adding}, $\widehat{G}'$ admits a 10/3-coloring $\phi$. By \Cref{lem:counting_partial_coloringK3}, $u$ has 3 available colors with respect to $\phi(x)$ and $\phi(y)$. Together with \Cref{lem:2-vertex}, $\phi$ can be extended to $\widehat{G}$, which is a contradiction. Thus $xy\in \widehat{G}$. 
	
	Let $\Gamma$ be the set consisting of  $u,x,y$  and all their 2-neighbors. We claim that any $10/3$-coloring $\phi$ of $\widehat{G}-\Gamma$ is extendable to $\widehat{G}$.  If $d(x)=4$, then two neighbors other than $u, y$ are 2-vertices, in which case, $\phi$ forbids at most two colors on $x$.  If $d(x)=3$, then $\phi$ forbids five colors on $x$.  Considering the same situation on $y$, the list $\mathcal{L}$ of available colors for $uxy$ satisfy $|L(x)|, |L(y)|\geq 5$, and $|L(u)|\geq 8$. It follows from  \Cref{lem:NegativeTirnagle} that for any such list $\mathcal{L}$, $\widehat{G}$ is $\mathcal{L}$-colorable.
\end{proof}

\begin{lemma} \label{lem:Counting_r(x)}
	Let $r(x)$ be the number of 2-neighbors and poor 3-neighbors and poor pseudo neighbors of $x$.
	\begin{itemize}
		\item if $d(x)=4$, then $r(x)\leq 2$;
		\item if $d(x)=5$, then $r(x)\leq 4$, moreover, $r(x)\leq 3$ if $x$ has a poor pseudo neighbor;
		\item if $d(x)=6$, then $r(x)\leq 5$;
		\item if $d(x)\ge 7$, then $r(x)\leq d(x)$.
\end{itemize}
\end{lemma}

\begin{proof}
	 Recall that $x$ and $y$ are pseudo neighbors meaning that there exists a subgraph $\widehat{H}\cong \widehat{T}$ such that $x$ and $y$ form one of the pseudo-neighbor pairs $\{v_{1},v_{3}\}$ or $\{v_{1},v_{4}\}$ in $\widehat{H}$. First we claim that each vertex has at most one poor pseudo neighbor. Towards a contradiction, assume that $x$ has two poor pseudo neighbors $u$ and $v$. Assume $x$, and $u$ ($x$, and $v$, respectively) are contained in a copy of $\widehat{T}$, named $\widehat{T}_u$ ($\widehat{T}_v$). Since each vertex has at most one pseudo neighbor in a copy of $\widehat{T}$, $v\notin \widehat{T}_u$. 
	We have $6 \leq |V(\widehat{T}_u \cup \widehat{T}_v)| \leq 9$. As $\delta(\widehat{T})=2$, the number of edges in $\widehat{T}_u \cup \widehat{T}_v$ grows by at least $2$ for every vertex added. Thus $\widehat{T}_u\cup \widehat{T}_v$ has at least $(2|V(\widehat{T}_u \cup \widehat{T}_v)|-3)$ edges. Hence
	$$p(\widehat{T}_u \cup \widehat{T}_v)\leq 3|V(\widehat{T}_u \cup \widehat{T}_v)|-2\times(2|V(\widehat{T}_u \cup \widehat{T}_v)|-3)=|V(\widehat{T}_u \cup \widehat{T}_v)|-6.$$
	
	By \Cref{lem:key}, $|V(\widehat{T}_u\cup \widehat{T}_v)|=6$ and $\widehat{G}=\widehat{T}_u\cup \widehat{T}_v$. However, one of the 2-neighbors of the poor 3-vertices $u$ and $v$ is not contained in $\widehat{T}_u\cup \widehat{T}_v$, which contradicts $\widehat{G}=\widehat{T}_u\cup \widehat{T}_v$.
	
	Assume $x$ has a poor pseudo neighbor $y$. Observe that $x\neq v_{1}$ as none of pseudo neighbors of $v_{1}$ in $\widehat{T}$ can be a poor vertex. Then $y$ is the degree 2 vertex in $\widehat{T}$. Note that  $d_{\widehat{T}}(y)=2$ and $d_{\widehat{G}}(y)=3$. Thus at least three neighbors of $x$ are non-poor by the structure of $\widehat{T}$. We have $r(x)\leq d(x)-2$. So we may assume that $x$ has no poor pseudo neighbor. 
	
	If $x$ has no poor 3-neighbor, then the claim follows from \Cref{lem:TwoNeighborCount}. Assume $x$ has a poor 3-neighbor $u$ with $N(u)=\{v, x, y\}$ and $d_G(v)=2$. By \Cref{lem:3_vertex_neighbors}, either $xy\in \widehat{G}$ or $\{x,y,u,w_1,w_2\}$ induces a copy of $\widehat{T}$. If $\{x,y,u,w_1,w_2\}$ induces a copy of $\widehat{T}$, then $r(x)\leq d(v)-2$ because of the structure of $\widehat{T}$. 
	
	Now we assume for each such poor 3-neighbor $u$ of $x$, $xy\in E(\widehat{G})$. As $y$ is neither a poor 3-vertex nor a 2-vertex, $r(x)\leq d(x)-1$. It remains to consider the case $d(x)=4$.

	Assume $d_G(x)=4$. Suppose $r(x)\in \{3, 4\}$.  We know that $xy\in E(\widehat{G})$, and by \Cref{lem:32-edge}, $x$ is rich. In other words, $x$ has at most one 2-neighbor. Thus $x$ has at least one more poor 3-neighbor $w$ other than $u$. Since $x$ is not a cut-vertex, $xyu$ and $xyw$ are triangles. By \Cref{lem:PairTriangles}, $\{u,v,x,y,w\}$ induces a copy of $\widehat{T}$. Without loss of generality, assume $wv$ is the only positive edge in $\widehat{G}[u,v,x,y,w]$. As $r(x)\ge 3$, the neighbor of $x$ other than $u,w,y$ is a 2-vertex, called $v'$. By criticality, let $\phi$ be a coloring of $\widehat{G}-\{u,v,x,w,v'\}$.  Without loss of generality, assume $\phi(y)=-3$. By  \Cref{obs:counting_partial_coloring-K2} and \Cref{lem:2-vertex}, it suffices to prove that the negative 4-cycle $uxwv$  is $\mathcal{L}$-colorable, where $L(u)=L(w)=\{1,2,3,4,5\}$, $L(v)=\{ \pm 1,\pm2, \pm 3, \pm4, \pm5\}$, $L(x)\subseteq \{1,2,3,4,5\}$  and $|L(x)|\ge 4$. 
	Depending on whether $1\in L(x)$ or not, one of the followings is the desired coloring:
	\begin{itemize}
	\item $\phi(u)=1$, $\phi(w)=2$, $\phi(x)=5$ and $\phi(v)=4$;
	\item $\phi(u)=4$, $\phi(w)=5$, $\phi(x)=1$ and $\phi(v)=-2$.\qedhere
	\end{itemize} \end{proof}

\section{Discharging}
In \Cref{lem:TwoNeighborCount}, \ref{lem:32-edge}, \ref{lem:Cubic}, \ref{lem:Delta4}, and  \ref{lem:Counting_r(x)}, we have essentially presented a list of forbidden configurations. For the final step of the proof, we use discharging method to derive a contradiction. 
Set the \emph{initial charge} $ch(v)=d(v)-3$ for every vertex $v$.

The discharging rules are as follows:
\begin{itemize}
	\item[(\bf{R1})] Every neighbor of a 2-vertex gives  $1/2$ to the 2-vertex.
	\item[(\bf{R2})] Every rich vertex gives $1/2$ to any poor 3-neighbor and  any poor pseudo-neighbor.
\end{itemize}

Let $ch^*$ denote the final charge after performing the discharging process. To contradict $p(\widehat{G})\ge 0$, it suffices to show that the final charge of each vertex $v$ is nonnegative and at least one vertex has a positive charge.

Assume $d(v)=2$. By (R1), $ch^*(v)=2-3+2\times1/2=0$.

Assume $d(v)=3$. It follows from \Cref{lem:TwoNeighborCount} that any $3$-vertex has at most one 2-neighbor. If $v$ has no 2-neighbor, then final charge of $v$ is 0. If $v$ is poor (has one 2-neighbor), then $v$ gives $1/2$ to the 2-neighbor (R1). But each rich neighbor and pseudo neighbor gives $1/2$ to the poor 3-vertex $v$ (R2). By \Cref{lem:32-edge}, $v$ has at least two neighbors or pseudo neighbors which are rich. Thus the final charge of any poor 3-vertex is more than 0. 

Assume $d(v)=4$. It follows from \Cref{lem:Counting_r(x)} that $r(v)\leq 2$. Thus $ch^*(v)\ge 4-3-2\times 1/2\ge 0$. 

Assume $d(v)= 5$. By \Cref{lem:Counting_r(x)}, the final charge of any  5-vertex is nonnegative. 

Assume $d(v)= 6$. It follows \Cref{lem:Counting_r(x)} that  $r(v)\leq 5$. Thus $ch^*(v)\ge 6-3-5\times 1/2> 0$. 

Assume $d(v)\ge 7$. By \Cref{lem:Counting_r(x)}, $ch^*(v)\ge d(v)-3-d(v)\times 1/2> 0$.  

As the final charge of any poor $3$-vertex and any $6^+$-vertex is bigger than 0, we may assume $\widehat{G}$ has no $(3,2)$-edge and $\Delta(G)\leq 5$. Since a rich 4-vertex has a positive final charge, and by \Cref{lem:Cubic} and \Cref{lem:Delta4}, we conclude that $\Delta(\widehat{G})= 5$ and $\widehat{G}$ has no poor $3$-vertex. Assume $d_G(v)=5$. 
Then by \Cref{lem:Counting_r(x)}, $r(v)\leq 3$. So $ch^*(v)\ge 1/2$ and we are done.

\section{Conclusion}
The circular chromatic number of a signed graph is a recent definition which extends similar notions from graphs and is a refinement of the notions of balanced coloring and $0$-free coloring. The specific value of $10/3$ is the best upper bound for the circular chromatic number of signed $K_{4}$-minor free graphs, see \cite{PZ22}. For some similar recent results see \cite{KNNW23}, \cite{ZZ23}.

An extension of the Brooks' theorem for $0$-free coloring of graphs of even maximum degree is given in \cite{MRS16}. It remains an open problem to give the best upper bound for the circular chromatic number of a signed (simple) graphs of a given odd maximum degree. The question appears rather challenging. 

In this work we addressed the first case of the problem, showing that any signed (simple) subcubic graph admits a circular $10/3$-coloring unless it is isomorphic to $(K_4, -)$. There are examples of signed (simple) subcubic graphs whose circular chromatic number is $10/3$; we have already seen that the signed graph $\widehat{T}$ has circular chromatic number $\frac{10}{3}$. Another example is a special signature on the Petersen graph as depicted in \Cref{fig:Petersen}. We note that this signed graph is of girth 5. We do not know of a simple proof for the fact that this signed graph does not admit any circular $r$-coloring for $r<\frac{10}{3}$. A proof that it does not admits a circular 3-coloring is given in \cite{BHNSW24}. One can then use the argument of tight cycle for the circular chromatic number to show that values between $3$ and $\frac{10}{3}$ are not possible either.

	\begin{figure}[h!]
	\centering
	\begin{tikzpicture}[scale=2.3, baseline=(current bounding box.center),
		node style/.style={circle,draw,inner sep=0.3pt, minimum size=3mm},
		dashed edge/.style={blue, dashed, thick},
		edge/.style={red, thick}
		]
		
		\foreach \i in {0,...,8} {
			\node[node style] (v\i) at ({90 - 40 * \i}:1) {}; 
		}
		
		\node[node style] (c) at (0,0) {};
		
		\foreach \i in {0,...,8} {
			\pgfmathtruncatemacro{\j}{mod(\i+1,9)}
			\draw[red] (v\i) -- (v\j);
		}
		
		\foreach \i in {0,3,6} {
			\draw[red] (c) -- (v\i);
		}
		
		\draw[blue, dash pattern=on .41mm off .41mm] (v1) -- (v5);
		\draw[blue, dash pattern=on .41mm off .41mm] (v2) -- (v7);
		\draw[blue, dash pattern=on .41mm off .41mm] (v4) -- (v8);
		
	\end{tikzpicture}
	\caption{The signed Petersen graph}
	\label{fig:Petersen}
\end{figure}
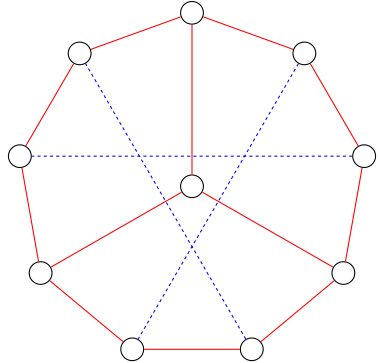

To prove the Brooks type theorem for circular coloring of signed graphs of maximum degree 3, we used the potential technique, proving that any signed simple graph which does not admit a circular $\frac{10}{3}$-coloring  contains either a $(K_4, -)$  or a subgraph on $n$ vertices with at least $\frac{3n+1}{2}$ edges, thus in particular a vertex of degree at least 4. This lower bound of $\frac{3n+1}{2}$ does not seem to be the tight value and it remains an open question to find the best lower bound. 

The sparest examples of $\frac{10}{3}$-critical signed graphs we know are basically $\frac{10}{3}$-critical graphs on $n=47k+8$ vertices and $e=84k+6$ edges. The construction is as follows. First consider an indicator construction $(I, s, t)$ which is built from $K_4$ by splitting one vertex to two vertices named $s$ and $t$, where one is of degree 2 and the other is of degree 1, see \Cref{fig:Indicator}.
Observe that $I$ admits a circular $\frac{10}{3}$-coloring but in any such a coloring the vertices $s$ and $t$ must be at distance at least $\frac{2}{3}$.

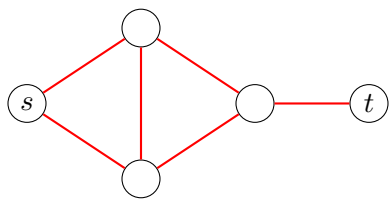
\begin{figure}
	\centering
	\begin{tikzpicture}[scale=0.5,
		baseline=(current bounding box.center),
		node style/.style={circle,draw,inner sep=0.5pt, minimum size=5mm},
		dashed edge/.style={blue, dashed, thick},
		edge/.style={red, thick}
		]
		
		\node[node style] (s) at (0, 0) {$s$};
		\node[node style] (x1) at (3, 2) {};
		\node[node style] (x2) at (3, -2) {};
		\node[node style] (x3) at (6, 0) {};
		\node[node style] (t) at (9, 0) {$t$};
		
		\draw[edge] (s) -- (x1) -- (x2) -- (x3) -- (t);
		\draw[edge] (s) -- (x2);
		\draw[edge] (x1) -- (x3);
	\end{tikzpicture}
	\caption{The indicator construction $(I, s, t)$}
	\label{fig:Indicator}
\end{figure}

Let $G$ be a $6$-critical graph in the classic sense, that is a graph whose chromatic number is 6 but every proper subgraph is 5-colorable. Observe that with our notation for circular coloring this would be a 5-critical graph or a $K_5$-critical graph. The classic Haj\'{o}s' construction starting from $K_6$, gives a sequence of $6$-critical graphs where the $k^{\text th}$ graph is on $5k+1$ vertices and has $14k+1$ edges. The density ration of $\frac{14k+1}{5k+1}$ is proved to be the best lower bound for the density of $6$-critical graphs \cite{KY14JCTB}. Now build a graph $I(G)$ from $G$ by replacing each edge $uv$ with a copy of the gadget $(I, s,t)$ where $s$ and $t$ are identified with the end vertices $u$ and $v$ while the other three vertices are distinct for each edge. The resulting graph has $(5k+1)+3\times(14k+1)=47k+4$ vertices and $6\times(14k+1)=84k+6$ edges. Any circular $\frac{10}{3}$-coloring of $I(G)$, by projecting the circle into a circle of circumference $5$, can be transformed to a circular 5-coloring of $G$. Thus $I(G)$ does not admit a $\frac{10}{3}$-coloring. We can verify that removing any edge we can get a $\frac{10}{3}$-coloring, concluding that $I(G)$ is $\frac{10}{3}$-critical, it has the density of $\frac{84k+6}{47k+4}$.     

Our final note is that the values of $\frac{10}{3}$ for circular chromatic number, which was focus of this work is among special numbers of the form $\frac{4q-2}{q}$. These values are of special interest because  $K_{4q-2;q}^s$ is the signed complete graph with positive loops where each vertex has $2q-1$ positive and $2q-1$ negative neighbors. Our focus in this work was the case $q=3$. The main focus of \cite{BHNSW24} has been the case $q=2$ of this sequence. The limit of this sequence is about circular 4-coloring which relates to signature packing problem and is studied in \cite{NY25+}.

\section*{Acknowledgment} This work has received support under the program ``Investissement d'Avenir" launched by the French Government and implemented by ANR, with the reference ``ANR‐18‐IdEx‐0001" as part of its program ``Emergence".

\bibliographystyle{plain}
\bibliography{references}

\end{document}